\documentclass[11pt]{amsart}

\usepackage{amssymb}
\usepackage{amsmath}
\usepackage{enumerate}
\usepackage{amsfonts}
\usepackage{color}
\usepackage{latexsym}

\usepackage[pdftex]{graphicx}
\DeclareGraphicsExtensions{.pdf,.png,.jpg}
\usepackage{xcolor}
\usepackage{tikz}
\usetikzlibrary{arrows}

\newtheorem{thm}{Theorem}[section]
\newtheorem{coro}[thm]{Corollary}
\newtheorem{lema}[thm]{Lemma}
\newtheorem{propo}[thm]{Proposition}
\newtheorem{remark}[thm]{Remark}

\numberwithin{equation}{section}

\DeclareMathOperator{\dist}{dist}
\DeclareMathOperator{\domain}{Dom}

\renewcommand{\a}{\alpha}

\renewcommand{\b}{\beta}

\allowdisplaybreaks

\author[\'O. Ciaurri]{\'Oscar Ciaurri}
\address{Departamento de Matem\'aticas y Computaci\'on,
         Universidad de La Rioja \newline
         26004 Logro\~no, Spain}
\email{oscar.ciaurri@unirioja.es}

\author[A. Nowak]{Adam Nowak}
\address{Institute of Mathematics,
Polish Academy of Sciences \newline
\'Sniadeckich 8,
00-656 Warszawa, Poland}
\email{adam.nowak@impan.pl}

\author[L. Roncal]{Luz Roncal}
\address{BCAM - Basque Center for Applied Mathematics, \newline
Alameda de Mazarredo 14,
48009 Bilbao, Spain}
\email{lroncal@bcamath.org}

\thanks{The first and third-named authors were supported by the grant MTM2015-65888-C4-4-P
from Spanish Government. The third-named author was also supported by the Basque Government through the
BERC 2014--2017 program and by Spanish Ministry of Economy and Competitiveness MINECO:
BCAM Severo Ochoa excellence accreditation SEV-2013-0323.}

\keywords{Spherical Radon transform, spherical mean, radial function, Hankel transform,
	Legendre function, kernel estimate, mixed norm estimate, two-weight estimate, wave equation,
	Euler-Poisson-Darboux equation, axially symmetric solution, Strichartz estimate}

\subjclass[2010]{Primary: 44A12; Secondary: 42B37, 35L15, 35B07, 35L05,35Q05.}

\begin{document}

\title[Mixed norm inequalities for spherical means]{Two-weight mixed norm estimates for\\
	a generalized spherical mean Radon transform \\ acting on radial functions}

\begin{abstract}
We investigate a generalized spherical means operator,
viz.\ generalized spherical mean Radon transform, acting on radial functions.
We establish an integral representation of this operator and find precise
estimates of the corresponding kernel.
As the main result, we prove two-weight mixed norm estimates for the integral operator, with
general power weights involved. This leads to weighted Strichartz type estimates for solutions
to certain Cauchy problems for classical Euler-Poisson-Darboux and wave equations with radial initial data.
\end{abstract}

\maketitle

\section{Introduction and preliminaries}

The spherical mean Radon transform in $\mathbb{R}^n\times \mathbb{R}_+$, $n \ge 1$, of a suitable function
$f$ is given by
$$
Mf(x,t) = \int_{S^{n-1}} f(x-ty)\, d\sigma(y), \qquad (x,t) \in \mathbb{R}^n \times \mathbb{R}_+,
$$
where $d\sigma$ is the normalized (probabilistic) measure on the unit sphere $S^{n-1} \subset \mathbb{R}^n$.
This operator returns the mean value of $f$ on the sphere centered at $x$ and of radius $t$.

The spherical means $M$ are of great importance in analysis and have been widely studied, due to
interest on their own right, their connections with a number of classical initial-value PDE problems,
as well as applications in physical/practical problems. The latter pertains to thermoacoustic and
photoacoustic tomography, among others, where inverse problems for $M$ play a crucial role.
A restriction of the action of $M$ to functions radially symmetric in the spatial variable is still
of interest for the very same reasons as just mentioned, but in various aspects admits more explicit
analysis and therefore, in this situation, one is able to obtain sharper or even optimal results.
It is by no means possible to give here a reasonably complete account of the results on $M$ obtained so far.
Thus we limit ourselves to mention only a few of them that inspired this work.

The study of $L^p$ estimates for the maximal operator $M^*f = \sup_{t>0}|Mf(\cdot,t)|$ was initiated
by Stein \cite{Stein}. He proved that $M^*$ is bounded on $L^p(\mathbb{R}^n)$ if and only if
$p > n/(n-1)$, provided that $n \ge 3$. Later Bourgain \cite{Bourgain} showed that the same holds
for $n=2$. Duoandikoetxea and Vega \cite{DuoVega} investigated weighted inequalities for $M^*$ and its
dyadic variant. Restricting to radial functions, Leckband \cite{Leckband} proved an endpoint result for
$M^*$, namely that it is of restricted weak type $(p,p)$ for $p=n/(n-1)$, $n \ge 2$. Recently, still
in the radial case, Duoandikoetxea, Moyua and Oruetxebarria \cite{DuoMoyuaOrue} obtained weighted
estimates for $M^*$ that are sharp for power weights.

A generalization of $M$ arises naturally in connection with a Cauchy problem for the classical
Euler-Poisson-Darboux (EPD in short) equation, see e.g.\ \cite{We,Brest} and references therein.
One considers the transformation
$$
M^{\b}f(x,t) = \mathcal{F}^{-1}\big( m_{\b}(t|\cdot|) \mathcal{F}f\big)(x),
$$
where $\mathcal{F}$ is the Fourier transform in $\mathbb{R}^n$ and the function defining the multiplier is
given by
$$
m_{\b}(s) = 2^{\b+n/2-1} \Gamma(\b+n/2) \frac{J_{\b+n/2-1}(s)}{s^{\b+n/2-1}}, \qquad s > 0,
$$
with $J_{\nu}$ denoting the (oscillating) Bessel function of the first kind and order $\nu$.
Given $t>0$, the operator $f \mapsto M^{\b}f(\cdot,t)$ extends meromorphically to all complex
$\b$ with poles at $\b = -n/2, -n/2-1, -n/2-2, \ldots$. 
For $\b=0$ one recovers the spherical means, i.e.\ $M^0=M$.
When $\b>0$, there is an integral representation (cf.\ \cite[p.\,171]{SW})
$$
M^{\b}f(x,t) = \frac{\Gamma(\b+n/2)}{\pi^{n/2}\Gamma(\b)} \int_{|y|<1} \big( 1-|y|^2\big)^{\b-1} f(x-ty)\, dy,
$$
the integration being over the unit ball in $\mathbb{R}^n$. Moreover, one can represent $M^{\b}$ in terms of $M$
via an Erd\'elyi-Kober fractional integral, see \cite{Rubin} for details.

Essentially the same generalization $M^{\b}$ was considered by Stein \cite{Stein}, where he obtained $L^p$
norm estimates for the associated maximal operator. In the same paper, Stein brings to readers' attention an important
link between $M^{\b}$ and a Cauchy problem for the classical wave equation.
Noteworthy, any solution to a general initial-value problem for the wave equation can be expressed
in terms of $M^0=M$ and its time derivatives only, at least in odd dimensions $n \ge 3$,
see e.g.\ \cite{CoCoSt} and references given there.

The aim of this paper is to study $M^{\b}$ acting on radially symmetric functions.
Such a transformation can be viewed as a family of operators $\{\mathcal{M}_t^{\a,\b}: t >0\}$
($\a$ being a parameter depending on $n$, to be specified in a moment) acting on profile functions
defined on $\mathbb{R}_+$. Then $\mathcal{M}_t^{\a,\b}$ can be expressed
in terms of the (modified) Hankel transform $\mathcal{H}_{\a}$. The latter is defined for $\a>-1$ and
suitable functions $f$ on $\mathbb{R}_+$ by
$$
\mathcal{H}_{\a}f(x) = \int_{0}^{\infty} f(y) \frac{J_{\a}(xy)}{(xy)^{\a}}\, d\mu_{\a}(y), \qquad x > 0,
$$
where $d\mu_{\a}(y) = y^{2\a+1}dy$. It is well known that $\mathcal{H}_{\a}$ extends to an isometry on
$L^2(\mathbb{R}_+,d\mu_{\a})$ which satisfies $\mathcal{H}_{\a}^{-1} = \mathcal{H}_{\a}$.
For $\a = n/2-1$ the (modified) Hankel transform corresponds to the Fourier transform in $\mathbb{R}^n$
acting on radial functions. Thus (for suitable $f$)
\begin{equation} \label{Handef}
\mathcal{M}_t^{\a,\b}f = \mathcal{H}_{\a}\big( m_{\a,\b}(t\cdot) \mathcal{H}_{\a}f\big), \qquad t >0,
\end{equation}
where the Hankel multipliers are defined by means of the function
$$
m_{\a,\b}(s) = 2^{\a+\b}\Gamma(\a+\b+1)\frac{J_{\a+\b}(s)}{s^{\a+\b}}, \qquad s > 0,
$$
and $\a$ has the form $\a=n/2-1$, $n \ge 1$. However, from analytic point of view,
there is no reason for restricting to the discrete set of $\a$. Accordingly, in what follows we allow a
continuous range $\a > -1$, actually the largest possible so that the Hankel transform is defined on the whole
$L^2(\mathbb{R}_+,d\mu_{\a})$.
There is also a deeper motivation for considering general $\a$, and this is related to certain
PDE problems involving Bessel operators rather than the standard Laplacian.
We shall always require \eqref{Handef} to be well defined on $L^2(\mathbb{R}_+,d\mu_{\a})$. This
happens exactly when $\a+\b \ge -1/2$, that is when $m_{\a,\b}$ is bounded (this fact follows from
basic asymptotics for the Bessel function, see below). If this is the case, then $\mathcal{M}_t^{\a,\b}$,
$t > 0$, are (uniformly) bounded operators on the $L^2$ space.

Our principal objective is to prove two-weight mixed norm $L^p-L^q(L^r_t)$ estimates 
for $\mathcal{M}_t^{\a,\b}$ with
possibly large classes of power weights admitted and possibly wide ranges of the parameters involved.
This is motivated by the limiting case $r=\infty$ corresponding to the maximal operator
$f \mapsto \sup_{t>0}|\mathcal{M}_t^{\a,\b}f|$ and the related investigations in
\cite{Stein,Bourgain,DuoVega,Leckband,DuoMoyuaOrue}.
However, comparing to $1 \le r < \infty$ the case $r=\infty$ requires a different and in fact more subtle
approach, therefore it will be treated in a separate paper.

For technical reasons it is much more convenient to work with an integral operator
$M_t^{\a,\b}$ that agrees with $\mathcal{M}_t^{\a,\b}$ in $L^2(\mathbb{R}_+,d\mu_{\a})$.
Thus our strategy is to switch to $M_t^{\a,\b}$ and then find precise estimates of the associated integral
kernel $K_t^{\a,\b}(x,z)$ in order to enable a direct and explicit analysis of the operator.
The latter relies on estimating first the norm of the kernel in power-weighted $L^r(dt)$, $1\le r < \infty$
and then showing two-weight mixed norm estimates for the resulting integral operator independent of the `time' variable $t$.

As illustrative applications of the mixed norm inequalities obtained, we derive weighted Strichartz type
estimates for solutions of certain initial-value problems for the EPD and wave equations, as well
as similar differential problems based on the one-dimensional Bessel operator
$L_{\a}=\frac{d^2}{dx^2} + \frac{2\a+1}x\frac{d}{dx}$. Such results seem to be desirable from the PDE theory
perspective.

An interesting aspect of our research is the behavior of the kernel $K_t^{\a,\b}(x,z)$. Perhaps a bit
surprisingly, there are half-lines and segments in the $(\a,\b)$ plane, where a kind of phase
shift occurs. More precisely, the behavior of the kernel is essentially different when $(\a,\b)$
belongs to those singular sets, comparing to the behavior in neighborhoods of those sets.
This phenomenon makes statements of the kernel estimates somewhat complicated.
Actually, something similar happens also in case of certain asymptotics for the Legendre functions
through which we express the kernel. The literature seems to tacitly omit those `singular' asymptotics.
This led us to derive them by ourselves, by means of combining various known facts and some computations.
Another topic that seems not to be covered properly by (at least standard) literature are zeros
of Legendre functions. Here we also had to work a bit by ourselves to derive what was needed for purposes
of this paper.

\subsection{Integral representation of the radial spherical means $\mathcal{M}_t^{\a,\b}$}
It turns out that $\mathcal{M}_t^{\a,\b}$ can be represented as a standard integral operator provided
that $\a+\b > -1/2$. In case $\a+\b=-1/2$ there is a singular integral representation, which is much more
subtle and not treated in this paper.

Define the kernel
$$
K^{\a,\b}_t(x,z)=2^{\a+\b}\Gamma(\a+\b+1) \int_0^{\infty}\frac{J_{\a+\b}(ty)}{(ty)^{\a+\b}}
\frac{J_{\a}(xy)}{(xy)^{\a}}\frac{J_{\a}(zy)}{(zy)^{\a}}d\mu_{\a}(y).
$$
As we shall see, $K_t^{\a,\b}(x,z)$ is well defined for $\a > -1$, $\a+\b > -1/2$, $t>0$ and (in general)
$t \neq |x-z|,t\neq x+z$. The integral here converges absolutely when $\a+\b > 1/2$, otherwise the
convergence at $\infty$ is only conditional, in the Riemann sense.
Notice that the kernel is homogeneous in the sense that
\begin{equation} \label{homo}
K_t^{\a,\b}(x,z) = \frac{1}{s^{2\a+2}} K^{\a,\b}_{t/s}\Big(\frac{x}s,\frac{z}s\Big), \qquad s,t,x,z > 0.
\end{equation}
This property is of importance
from the point of view of analysis related to the operator we now define.

For each $t>0$ consider the integral operator
$$
M_t^{\a,\b}f(x)=\int_0^{\infty}K^{\a,\b}_t(x,z)f(z)\, d\mu_{\a}(z)
$$
on its natural domain $\domain M_t^{\a,\b}$ consisting of all those $f$ for which the above integral
converges absolutely for a.a.\ $x>0$. The following result gives a preliminary link between
$\mathcal{M}_t^{\a,\b}$ and $M_t^{\a,\b}$.
\begin{propo} \label{prop:coinc}
Let $\a > -1$ and $\a+\b > -1/2$. Then, for each $t > 0$,
$C_c^{\infty}(0,\infty) \subset \domain M_t^{\a,\b}$ and
$$
\mathcal{M}_t^{\a,\b}f = M_t^{\a,\b}f, \qquad f \in C_c^{\infty}(0,\infty).
$$
\end{propo}

\begin{proof}
We shall use the basic asymptotics for the Bessel function,
$$
J_{\a}(u) \simeq u^{\a}, \quad u \to 0^+, \qquad J_{\a}(u) = \mathcal{O}(u^{-1/2}), \quad u \to \infty,
$$
together with the following fact: if $g = \mathcal{H}_{\a}f$ for some $f \in C_c^{\infty}(0,\infty)$,
then $g$ is continuous, $g(u) = \mathcal{O}(1)$ as $u \to 0^+$ and, for any fixed $k$,
$g(u) = \mathcal{O}(u^{-k})$ as $u \to \infty$; see \cite[Section 2.1]{NoSt}.

Let $f \in C_c^{\infty}(0,\infty)$.
Since, in view of what was said above,
$m_{\a,\b}(t\cdot) \mathcal{H}_{\a}f \in L^1(d\mu_{\a})$, and also $f \in L^1(d\mu_{\a})$,
we can write
\begin{align} \nonumber
\frac{\mathcal{M}_t^{\a,\b}f(x)}{2^{\a+\b}\Gamma(\a+\b+1)} & = \int_0^{\infty} \frac{J_{\a}(xy)}{(xy)^{\a}}
	\frac{J_{\a+\b}(ty)}{(ty)^{\a+\b}} \mathcal{H}_{\a}f(y)\, d\mu_{\a}(y) \\
& = \int_0^{\infty}\int_0^{\infty} \frac{J_{\a+\b}(ty)}{(ty)^{\a+\b}}
\frac{J_{\a}(xy)}{(xy)^{\a}}\frac{J_{\a}(zy)}{(zy)^{\a}} \, f(z)\, d\mu_{\a}(z)\, d\mu_{\a}(y),
\qquad t,x > 0. \label{o1}
\end{align}
The proof will be finished once we show that changing the order of integration in \eqref{o1} is legitimate.
It is easily seen that this is indeed the case when $\a+\b > 1/2$, since then the last double integral
converges absolutely and one can use Fubini's theorem. However, in case $\a+\b \le 1/2$ the situation
is more delicate because the integral defining $K_t^{\a,\b}(x,z)$ converges in the Riemann sense, but
not absolutely.

To proceed, we assume that $t,x>0$ are fixed, and $f$ is also fixed and its support is contained in
an interval $[A,B]$ with $0 < A < B < \infty$. Splitting the outer integral in \eqref{o1} we reduce the
problem to switching the order integration in
$$
\int_1^{\infty}\int_A^B \frac{J_{\a+\b}(ty)}{(ty)^{\a+\b}}
\frac{J_{\a}(xy)}{(xy)^{\a}}\frac{J_{\a}(zy)}{(zy)^{\a}} \, f(z)\, d\mu_{\a}(z)\, d\mu_{\a}(y).
$$

Next, we expand each of the three Bessel functions according to the large argument asymptotics,
see \cite[Chapter VII, \S 7$\cdot$21(1)]{Wat},
$$
J_{\nu}(w) = \sqrt{\frac{2}{\pi w}} \cos\bigg( w - \frac{\pi}2\Big(\nu+\frac{1}2\Big)\bigg) +
	\mathcal{O}\big(w^{-3/2}\big),
$$
valid for positive $w$ separated from $0$. This leads to a splitting of the integrand into
8 terms. All the resulting double integrals converge absolutely, except for one,
written here up to some factors that can be neglected in further analysis:
$$
\int_1^{\infty}\int_A^B \frac{\cos(ty-c_1) \cos(xy-c_2) \cos(zy-c_2)}{y^{\a+\b+1/2}} f(z)\, dz\, dy.
$$

In the last expression we can write the outer integral as a limit of integrals over bounded intervals
and exchange the order of integration due to absolute integrability. This means that our task reduces
to checking that one can pass with the limit under the integral sign in
$$
\lim_{N \to \infty} \int_A^B G_N(z) f(z)\, dz,
$$
where
$$
G_N(z) = \int_1^{N} \frac{\cos(ty-c_1) \cos(xy-c_2) \cos(zy-c_2)}{y^{\a+\b+1/2}}\, dy.
$$
If we now show that the sequence $\{|G_N|\}$ is controlled by an integrable function over $[A,B]$, then
the desired conclusion will follow from the dominated convergence theorem.

To continue, we invoke the product-to-sum formula
$$
\cos \theta_1 \cos\theta_2 \cos\theta_3
= \frac{1}{8} \sum_{(e_1,e_2,e_3) \in \{-1,1\}^3} 
\cos(e_1\theta_1+e_2\theta_2+e_3\theta_3)
$$
getting
\begin{align*}
& 4 \cos(ty-c_1) \cos(xy-c_2) \cos(zy-c_2) \\
& \quad = \cos\big( y(t+x+z)-c_1-2c_2\big) + \cos\big(y(t+x-z)-c_1\big) \\
& \qquad + \cos\big(y(t-x+z)-c_1\big) + \cos\big(y(x+z-t)+c_1-2c_2\big).
\end{align*}
We now see that it is enough to verify the existence of an integrable over $[A,B]$ majorant of 
$\{|H_N(z)|\}$, where the new sequence is of the form
$$
H_N(z) = \int_1^N \frac{\cos\big( y(z-D)+C\big)}{y^{\lambda}} \, dy,
$$
with $0 < \lambda \le 1$ fixed and $C,D \in \mathbb{R}$ also fixed
(of course, it may happen that $D \in [A,B]$).

We treat here a simplified model situation, which contains the heart of the matter.
The general case requires then some elementary technical adjustments, which are left to the reader.
Let $A=C=D=0$. Then, changing the variable of integration, we get
$$
H_N(z) = \int_1^{N} \frac{\cos(yz)}{y^{\lambda}}\, dy 
	= z^{\lambda-1} \int_z^{Nz} \frac{\cos s}{s^{\lambda}}\, ds.
$$
If $\lambda < 1$, the last integral stays bounded when $N$ and $z$ vary,
since $s^{-\lambda}\cos s$ is integrable over $(0,\infty)$ (at $\infty$ in the Riemann sense only).
Thus we see that the required majorant is $H(z) = c z^{\lambda-1}$.
When $\lambda =1$ we split the last integral with respect to the point $1 \wedge (Nz)$ and then easily
see that in this case the majorant is $H(z) = c (1+\log^+\frac{1}z)$.
\end{proof}

Later, in Section \ref{sec:L2}, we will see that for each $t>0$, 
$L^2(\mathbb{R}_+,d\mu_{\a}) \subset \domain M_t^{\a,\b}$ and $M_t^{\a,\b}$
is bounded on $L^2(\mathbb{R}_+,d\mu_{\a})$. This together with Proposition \ref{prop:coinc} implies
that $\mathcal{M}_t^{\a,\b}$ and $M_t^{\a,\b}$ coincide as operators acting on $L^2(\mathbb{R}_+,d\mu_{\a})$.

\subsection{Structure of the paper and notation}
The rest of the paper is organized as follows. In Section \ref{sec:estI},
to feel flavor of the problem and gain a better intuition, we find sharp estimates of
the integral kernel of $M_t^{\a,\b}$ in the uncomplicated case when $\a > -1/2$ and $\b > 0$, see Theorem \ref{thm:est_ker}.
This is done by employing a relatively simple positive integral representation for the triple Bessel
function integral entering the kernel. Then, with the aid of sharp bounds for certain elementary integrals
(see Lemmas \ref{est_I} and \ref{est_J}), the result follows in a rather straightforward manner.
We also look at a few special cases of the parameters $\a,\b$ in which the kernel $K_t^{\a,\b}(x,z)$ is
totally computable. This reveals, in particular, that one has to be careful when it comes to values
of the kernel related to the singular surfaces $t=|x-z|$ and $t=x+z$, even if those values are finite.
In Section \ref{sec:estII} we estimate the kernel $K_t^{\a,\b}(x,z)$ in the general case when $\a > -1$
and $\b > -\a-1/2$. Here the strategy is to express the triple Bessel function integral via suitable
Legendre functions and then estimate the resulting expressions by means of Legendre functions asymptotics.
Since the latter seem to be incomplete, at least in a standard literature, we derive the missing cases
by ourselves, using known facts and formulas and explicit computations. Another important issue
we study in this section is presence or lack of zeros of the Legendre functions, since in the latter
case the estimates of the kernel we get are in fact sharp. Our main result on the behavior of
$K_t^{\a,\b}(x,z)$ is stated in Theorem \ref{thm:kerest2}.
In Section \ref{sec:L2} we prove that for each $t > 0$ the integral
operator $M_t^{\a,\b}$ is well defined and bounded on $L^2(\mathbb{R}_+,d\mu_{\a})$. Consequently, the $L^2$-coincidence
between $\mathcal{M}_t^{\a,\b}$ and $M_t^{\a,\b}$ is established, see Proposition \ref{cor:L2}.
In Section \ref{sec:test} we estimate the norm of the kernel $K^{\a,\b}_t(x,z)$ in power-weighted
$L^r(dt)$. The bounds we get are fairly precise in general, and sharp in many cases;
see Theorem \ref{thm:norm_est}. 
In Section \ref{sec:mixed} we state the main result of the paper, that is the two-weight mixed norm
estimate for $M_t^{\a,\b}$ which is contained in Theorem \ref{thm:main}.
This is preceded by a sharp analysis of an auxiliary integral operator emerging from the precise
absolute estimates for the kernel $K_t^{\a,\b}(x,z)$ obtained previously in Theorem \ref{thm:kerest2}.
Finally, Section \ref{sec:appl} is devoted to applications of the mixed norm estimates.
These pertain to weighted Strichartz type estimates for solutions to certain radial initial-value
problems for classical Euler-Poisson-Darboux and wave equations, as well as Bessel operator based
counterparts of these equations.

For readers' convenience we enclose a detailed table of contents and a list of figures.
\tableofcontents
\listoffigures
Throughout the paper we use a fairly standard notation. Thus $\mathbb{R}_+=(0,\infty)$.
The symbols ``$\vee$'' and ``$\wedge$'' mean the operations of taking maximum and minimum, respectively.
We write $X\lesssim Y$ to indicate that $X\leq CY$ with a positive constant $C$
independent of significant quantities. We shall write $X \simeq Y$ when simultaneously
$X \lesssim Y$ and $Y \lesssim X$.

For the sake of brevity, we shall omit $\mathbb{R}_+$ when denoting $L^p$ spaces related to the measure space
$(\mathbb{R}_+,d\mu_{\a})$.
Given a non-negative weight $w$, we denote by $L^p(w^pd\mu_{\a})$ the weighted $L^p$ space with respect
to the measure $\mu_{\a}$. This means that $f \in L^p(w^pd\mu_{\a})$ if and only if $wf \in L^p(d\mu_{\a})$.
By convention, $L^{\infty}(w^{\infty}d\mu_{\a})$ consists of all measurable functions $f$ such that
$wf$ is essentially bounded on $\mathbb{R}_+$ and the norm of $f$ in that space is $\|wf\|_{\infty}$.
We write $L^p_{\textrm{rad}}(\ldots)$ for the subspace of $L^p(\ldots)$ consisting of radial functions.
As usual, for $1 \le p \le \infty$, $p'$ denotes its conjugate exponent, $1/p+1/p'=1$.

\section{Pointwise kernel estimates I: a special case} \label{sec:estI}

In this section we prove, by elementary methods, sharp estimates of the kernel $K_t^{\a,\b}(x,z)$
in case $\a > -1/2$ and $\b > 0$.
For such parameters the kernel is given by means of a well-studied generalization of the
Weber-Schafheitlin integral. More precisely,
formula \cite[Chapter XII, \S13$\cdot$46(1)]{Wat} implies, for $t,x,z > 0$ such that 
$t \neq |x-z|$ and $t\neq x+z$,
\begin{equation} \label{ifK}
K_t^{\a,\b}(x,z)=c_{\a,\b}\,t^{-2(\a+\b)}\int_0^A \big(t^2-x^2-z^2+2xz\cos\theta\big)^{\b-1}
	\sin^{2\a}\theta\,d\theta.
\end{equation}
Here $c_{\a,\b} = 2\Gamma(\a+\b+1)/(\sqrt{\pi}\Gamma(\a+1/2)\Gamma(\b))$ and
$$
A=\begin{cases}
	0, & \quad t < |x-z|,\\
	\arccos \Big(\frac{x^2+z^2-t^2}{2xz}\Big), & \quad |x-z| < t < x+z,\\
	\pi, & \quad t > x+z.
	\end{cases}
$$
Notice that for $\a$ and $\b$ under consideration $K_t^{\a,\b}(x,z)$ is non-negative,
$K_t^{\a,\b}(x,z) = 0$ if $t < |x-z|$, and $K_t^{\a,\b}(x,z)>0$ when $t > |x-z|$. 

\subsection{Two simple technical results}
We need precise estimates of the following integrals:
\begin{align*}
I_{\a,\gamma}(B) & := \int_{-1}^1 (1-Bs)^{\gamma} (1-s^2)^{\a-1/2}\, ds, \qquad 0 \le B \le 1,\\
J_{\a,\b,\gamma}(D) & := \int_0^1 (D-s)^{\a-1/2} (1-s)^{\b-1} s^{\gamma} \, ds, \qquad D \ge 1.
\end{align*}

\begin{lema} \label{est_I}
Let $\a > -1/2$ and $\gamma \in \mathbb{R}$ be fixed. Then
\begin{equation*}
I_{\a,\gamma}(B)\simeq 
	\begin{cases}
		(1-B)^{\a+\gamma+1/2}, & \quad \a+\gamma + 1/2 < 0,\\
		1+\log\frac{1}{1-B}, & \quad \a+\gamma + 1/2 = 0,\\
		1, & \quad \a+\gamma+1/2 > 0,
	\end{cases}
\end{equation*}
uniformly in $0 \le B \le 1$.
\end{lema}

The case $B=1$ in the statement of Lemma \ref{est_I} and in the proof below should be understood in
the usual limiting sense.

\begin{proof}[Proof of Lemma \ref{est_I}]
Assume, to begin with, that $\gamma \ge 0$. Then the essential contribution to $I_{\a,\gamma}(B)$
comes from integration between $-1$ and $0$. Therefore, taking into account that $1-Bs \simeq 1$
and $1-s \simeq 1$ when $-1 < s < 0$, we can write
$$
I_{\a,\gamma}(B) \simeq \int_{-1}^0 (1-Bs)^{\gamma} (1-s^2)^{\a-1/2}\, ds 
	\simeq \int_{-1}^0 (1+s)^{\a-1/2}\, ds \simeq 1.
$$
This agrees with the asserted estimate, since $\gamma \ge 0$ implies $\a + \gamma + 1/2 > 0$.

Assume next that $\gamma < 0$. Now the essential contribution to $I_{\a,\gamma}(B)$ comes from
integration between $0$ and $1$, and we have
\begin{equation} \label{c1}
I_{\a,\gamma}(B) \simeq \int_0^1 (1-Bs)^{\gamma} (1-s)^{\a-1/2}\, ds.
\end{equation}
When $\a+\gamma + 1/2 > 0$, we use the straightforward bounds
$$
\int_0^1 (1-s)^{\a-1/2}\, ds \lesssim I_{\a,\gamma}(B) \lesssim \int_0^1 (1-s)^{\a+\gamma -1/2}\, ds
$$
to conclude that $I_{\a,\gamma}(B) \simeq 1$. Thus it remains to treat the case $\a+\gamma + 1/2 \le 0$.
Here we may assume that $B > 1/2$, since otherwise $I_{\a,\gamma}(B) \simeq 1$, as needed.
Changing the variable of integration $s = 1-\frac{1-B}{B}w$ in \eqref{c1}, and remembering that now
$B \simeq 1$, we get
$$
I_{\a,\gamma}(B) \simeq (1-B)^{\a+\gamma+1/2} \int_0^{B/(1-B)} (1+w)^{\gamma} w^{\a-1/2}\, dw.
$$
Denoting by $\widetilde{I}_{\a,\gamma}(B)$ the last integral, we see that
$$
\widetilde{I}_{\a,\gamma}(B) \simeq 1 + \int_1^{B/(1-B)} w^{\a+\gamma-1/2}\, dw
	\simeq \begin{cases}
						1, & \quad \a+\gamma +1/2 < 0,\\
						1 + \log\frac{B}{1-B}, & \quad \a+\gamma + 1/2 = 0, 
					\end{cases}
$$
where in the log case $\frac{B}{1-B}$ may be replaced by $\frac{1}{1-B}$. The conclusion follows.
\end{proof}

\begin{lema} \label{est_J}
Let $\a \in \mathbb{R}$, $\b > 0$ and $\gamma > -1$ be fixed. Then
$$
J_{\a,\b,\gamma}(D) \simeq D^{\a-1/2}, \qquad D \ge 2,
$$
and
$$
J_{\a,\b,\gamma}(D) \simeq  \begin{cases}
															(D-1)^{\a+\b -1/2}, & \quad \a+\b < 1/2, \\
															1 + \log\frac{1}{D-1}, & \quad \a+\b = 1/2, \\
															1, & \quad \a+\b > 1/2,
														\end{cases}
$$
uniformly in $1 \le D < 2$.
\end{lema}

The case $D=1$ in the statement of Lemma \ref{est_J} and in its proof is understood in the limiting sense.

\begin{proof}[Proof of Lemma \ref{est_J}]
Split the integral defining $J_{\a,\b,\gamma}(D)$ according to the intervals $(0,1/2)$ and $(1/2,1)$,
and denote the resulting integrals by $J^0_{\a,\b,\gamma}(D)$ and $J^1_{\a,\b,\gamma}(D)$, respectively.
Then we have
$$
J^0_{\a,\b,\gamma}(D) \simeq D^{\a-1/2} \int_0^{1/2} s^{\gamma}\, ds \simeq D^{\a-1/2}.
$$
For the complementary integral we write
$$
J^1_{\a,\b,\gamma}(D) \simeq \int_{1/2}^1 (D-s)^{\a-1/2} (1-s)^{\b-1}\, ds. 
$$
Changing now the variable of integration $s = 1- (D-1)w$ leads to
$$
J^1_{\a,\b,\gamma}(D) \simeq (D-1)^{\a+\b-1/2} \int_0^{\frac{1}{2(D-1)}} (1+w)^{\a-1/2} w^{\b-1}\, dw.
$$
For $D \ge 3/2$ it follows that
$$
J^1_{\a,\b,\gamma}(D) \simeq (D-1)^{\a+\b-1/2} \int_0^{\frac{1}{2(D-1)}} w^{\b-1}\, dw \simeq
	(D-1)^{\a-1/2} \simeq D^{\a-1/2}.
$$
On the other hand, if $1 \le D < 3/2$,
\begin{align*}
J^1_{\a,\b,\gamma}(D) & 
	\simeq (D-1)^{\a+\b-1/2} \Bigg[ 1 + \int_1^{\frac{1}{2(D-1)}} w^{\a+\b-3/2}\, dw \Bigg] \\
	& \simeq  \begin{cases}
						(D-1)^{\a+\b-1/2}, & \quad \a + \b  < 1/2, \\
						1 + \log\frac{1}{D-1}, & \quad \a + \b = 1/2, \\
						1, & \quad \a + \b  > 1/2.
					\end{cases}
\end{align*}
Combining the above estimates of $J^0_{\a,\b,\gamma}(D)$ and $J^1_{\a,\b,\gamma}(D)$ we arrive at the
desired conclusion.
\end{proof}

\subsection{Estimates of the kernel $K_t^{\a,\b}(x,z)$}
We are now ready to prove sharp estimates of $K_t^{\a,\b}(x,z)$.
Recall that the kernel vanishes in the region $\{t,x,z > 0 : t < |x-z|\}$. 

\begin{thm} \label{thm:est_ker}
Let $\a > -1/2$ and $\b > 0$ be fixed. Let $t,x,z > 0$. Then
\begin{align*}
K_t^{\a,\b}(x,z) & \simeq 
\frac{(xz)^{-\a-1/2}}{t^{2\a+2\b}} \big[ t^2-(x-z)^2\big]^{\a+\b-1/2}
		\begin{cases}
			\left( \frac{(x+z)^2-t^2}{xz}\right)^{\a+\b-1/2}, & \a + \b < 1/2, \\
			1 + \log\left(\frac{xz}{(x+z)^2-t^2}\right), & \a + \b = 1/2,\\
			1, & \a + \b > 1/2,		
		\end{cases}
\end{align*}
uniformly in $|x-z| < t < x+z$, and
$$
					K_t^{\a,\b}(x,z) \simeq
					\frac{1}{t^{2\a+2\b}} \big[ t^2-(x-z)^2\big]^{\b-1} 
						\begin{cases}
							\left( \frac{t^2-(x+z)^2}{t^2-(x-z)^2}\right)^{\a+\b-1/2}, & \a + \b < 1/2, \\
							1 + \log\left(\frac{t^2-(x-z)^2}{t^2-(x+z)^2}\right), & \a + \b = 1/2, \\
							1, & \a + \b > 1/2,
						\end{cases}						
$$
uniformly in $t > x+z$.
\end{thm}

Taking into account the relation
\begin{equation} \label{st5}
xz \simeq \big[t^2-(x-z)^2\big] \vee \big[ (x+z)^2-t^2\big], \qquad |x-z| < t < x+z,
\end{equation}
we see that the behavior of $K_t^{\a,\b}(x,z)$ depends on $x$ and $z$ only
through $|x-z|$ and $x+z$. Moreover, this behavior depends essentially on the distances from $t^2$ to
$(x-z)^2$ and $(x+z)^2$, and some singularities occur when any of them tends to zero.
It is perhaps interesting to observe that the bounds from Theorem \ref{thm:est_ker} can be written in
a more compact way as
\begin{align*}
K_t^{\a,\b}(x,z) & \simeq
t^{-2(\a+\b)} \Big[ \big(t^2-(x-z)^2\big) \vee \big| t^2-(x+z)^2\big| \Big]^{-\a-1/2} \\
& \qquad \times	
	\begin{cases}
		\big[ \big(t^2-(x-z)^2\big)\wedge \big|t^2-(x+z)^2\big|\big]^{\a+\b-1/2}, & \quad \a+\b < 1/2, \\
		1 + \log\frac{t^2-(x-z)^2}{[t^2-(x-z)^2]\wedge |t^2-(x+z)^2|}, & \quad \a+\b = 1/2, \\
		\big(t^2-(x-z)^2\big)^{\a+\b-1/2}, & \quad \a+\b > 1/2,
	\end{cases}
\end{align*}	
uniformly in $t > |x-z|$ such that $t \neq x+z$.

\begin{proof}[{Proof of Theorem \ref{thm:est_ker}}]
We distinguish three cases emerging from splitting the range of $t^2$ according to the points
$(x-z)^2$, $x^2+z^2$ and $(x+z)^2$. Observe that the middle point is the geometric center of the interval
defined by the other points as endpoints.

\noindent \textbf{Case 1. ${t^2 > (x+z)^2}$.} 
Changing the variable of integration $\cos\theta = -s$ in \eqref{ifK}, we get
\begin{align*}
K_t^{\a,\b}(x,z) & = c_{\a,\b}\, t^{-2(\a+\b)} \int_{-1}^1 \big( t^2-x^2-z^2-2xzs\big)^{\b-1} (1-s^2)^{\a-1/2}\, ds \\
	& = c_{\a,\b}\, t^{-2(\a+\b)} \big(t^2-x^2-z^2\big)^{\b-1}\, I_{\a,\b-1}\bigg(\frac{2xz}{t^2-x^2-z^2}\bigg).
\end{align*}
Now an application of Lemma \ref{est_I} gives
$$
K_t^{\a,\b}(x,z) \simeq t^{-2(\a+\b)} \big( t^2-x^2-z^2\big)^{\b-1}
		\begin{cases}
			\Big(\frac{t^2-(x+z)^2}{t^2-x^2-z^2}\Big)^{\a+\b-1/2}, & \quad \a+\b < 1/2,\\
			1+\log\frac{t^2-x^2-z^2}{t^2-(x+z)^2}, & \quad \a+\b = 1/2,\\
			1, & \quad \a+\b > 1/2.
		\end{cases}
$$
Since $t^2-x^2-z^2 \simeq t^2-(x-z)^2$, this is equivalent to the bounds of the theorem.

\noindent \textbf{Case 2. ${x^2+z^2 < t^2 < (x+z)^2}$.}
Changing the variable of integration $\cos\theta = 1 - \frac{t^2-(x-z)^2}{2xz}s$ in \eqref{ifK} and then
simplifying the resulting expression we see that
\begin{equation} \label{e2f}
	K_t^{\a,\b}(x,z) = c_{\a,\b}\, t^{-2(\a+\b)} (2xz)^{-2\a} \big[t^2-(x-z)^2\big]^{2\a+\b-1}
			J_{\a,\b,\a-1/2}\bigg( \frac{4xz}{t^2-(x-z)^2}\bigg).
\end{equation}
Since $t^2 > x^2+z^2$ is equivalent to $\frac{4xz}{t^2-(x-z)^2} < 2$, Lemma \ref{est_J} implies
$$
K_t^{\a,\b}(x,z) \simeq t^{-2(\a+\b)} (xz)^{-2\a} \big[ t^2 -(x-z)^2\big]^{2\a+\b-1}
		\begin{cases}
			\Big( \frac{(x+z)^2-t^2}{t^2-(x-z)^2}\Big)^{\a+\b-1/2}, & \quad \a+\b < 1/2, \\
			1 + \log\frac{t^2-(x-z)^2}{(x+z)^2-t^2}, & \quad \a+\b=1/2, \\
			1, & \quad \a+\b > 1/2,
		\end{cases}
$$
which with the aid of \eqref{st5} leads to the estimates of the theorem.

\noindent \textbf{Case 3. ${(x-z)^2 < t^2 \le x^2+z^2}$.}
In view of \eqref{st5}, the estimate in question can be stated simply as
$$
K_t^{\a,\b}(x,z) \simeq t^{-2(\a+\b)} (xz)^{-\a-1/2} \big[t^2-(x-z)^2\big]^{\a+\b-1/2}.
$$
But this is a straightforward consequence of \eqref{e2f} and Lemma \ref{est_J}, since
$t^2 \le x^2+z^2$ means that $\frac{4xz}{t^2-(x-z)^2} \ge 2$.

The proof of Theorem \ref{thm:est_ker} is complete. 
\end{proof}

\subsection{Some special elementary cases}
It is interesting to observe that for $\b=0$ the kernel $K_t^{\a,\b}(x,z)$ can be computed explicitly.
To this end let $\a > -1/2$ and $\b =0$. We have
$$
K_t^{\a,0}(x,z) = 2^{\a}\Gamma(\a+1)(txz)^{-\a} \int_0^{\infty} J_{\a}(ty) J_{\a}(xy) J_{\a}(zy)\, y^{1-\a}\, dy.
$$
In virtue of formula \cite[Chapter XII, \S13$\cdot$46(3)]{Wat},
see also \cite[Formula (14) on p.\,230]{Prudnikov2},
the kernel vanishes if either
$t < |x-z|$ or $t > x+z$, and for $|x-z|< t < x+z$ we have
\begin{equation} \label{kera0}
K_t^{\a,0}(x,z) = \frac{\Gamma(\a+1)}{\sqrt{\pi}2^{2\a-1}\Gamma(\a+1/2)} (txz)^{-2\a}
	\Big( \big[t^2-(x-z)^2\big] \big[(x+z)^2-t^2\big]\Big)^{\a-1/2}.
\end{equation}

In \cite{Leckband} and \cite{DuoMoyuaOrue} standard spherical means of radial functions 
in $\mathbb{R}^n$, $n \ge 2$, were considered. These means are represented via the one-dimensional kernel
\begin{equation} \label{kerL}
L_t^{n}(x,z) = \frac{4\Gamma(n/2)}{\Gamma((n-1)/2)\sqrt{\pi}} \bigg[ 
	\frac{2(b^2-t^2)^{1/2}(t^2-a^2)^{1/2}}{b^2-a^2}\bigg]^{n-3} \frac{t^{2-n}}{b^2-a^2},
\end{equation}
where $a = |x-z|$, $b=x+z$, $|x-z|< t < x+z$; the related measure of integration is $z^{n-1}dz$.
One can generalize \eqref{kerL} by letting $n=2\a+2$ and considering a continuous range $\a > -1/2$.
Then it is straightforward to check that
$$
L_t^{2\a+2}(x,z) = K_t^{\a,0}(x,z), \qquad |x-z| < t < x+z.
$$
Moreover, the measure $z^{n-1}dz$ becomes $d\mu_{\a}(z)$. Thus the kernel $K_t^{\a,\b}(x,z)$ generalizes
$L_t^{2\a+2}(x,z)$.

We now look closer at the more elementary case $\b=0$, $\a=1/2$, to see that one indeed has to be careful
when dealing with values of $K_t^{\a,\b}(x,z)$ in the singular cases $t=|x-z|$ and $t=x+z$ (anyway,
values of the kernel for those $t,x,z$ are irrelevant for our purposes). We have
$$
K_t^{1/2,0}(x,z) = \frac{2}{\pi} (txz)^{-1} \int_0^{\infty} \sin(ty) \sin(xy) \sin(zy) \,
	\frac{dy}y.
$$
Using \cite[Formula (3) on p.\,411]{Prudnikov1}, we find that
$$
K_t^{1/2,0}(x,z) = \frac{1}{2} (txz)^{-1}
	\begin{cases}
		0, & \quad |x-z| > t \;\; \textrm{or} \;\; t> x+z, \\
		1/2, & \quad t=|x-z| \;\; \textrm{or} \;\; t= x+z, \\
		1, & \quad |x-z| < t < x+z.
	\end{cases}
$$

Another example in this spirit is the case $\a=-1/2$, $\b=1$, computed by means of 
\cite[Formula (21) on pp.\,406--407]{Prudnikov1}:
$$
K_t^{-1/2,1}(x,z) = \frac{1}{t}
	\begin{cases}
		0, & \quad t < |x-z|,\\
		1/4, & \quad t=|x-z|,\\
		1/2, & \quad |x-z| < t < x+z,\\
		3/4, & \quad t= x+z,\\
		1, & \quad t > x+z.
	\end{cases}
$$

A careful reader probably have noticed that the explicit formulas just given are not consistent with the
estimates of Theorem \ref{thm:est_ker}. More precisely, a kind of phase shift occurs for $\a \ge 1/2$ and $\b=0$,
as well as for $\a=-1/2$ and $\b=1$. This interesting and perhaps a bit unexpected phenomenon will be fully
revealed in Theorem \ref{thm:kerest2}.

Further explicit formulas for the kernel are implicitly contained in Section \ref{ss:expl} below,
see Remark \ref{rem:ker_expl}.

\section{Pointwise kernel estimates II: general case} \label{sec:estII}

Recall that the kernel we are dealing with is
$$
K_t^{\a,\b}(x,z) = \frac{2^{\a+\b}\Gamma(\a+\b+1)}{t^{\a+\b}(xz)^{\a}} \int_0^{\infty} J_{\a+\b}(ty)
	J_{\a}(xy) J_{\a}(zy) y^{1-\a-\b}\, dy, \qquad t,x,z > 0. 
$$
In the previous section we found sharp estimates of this kernel when $\a > -1/2$ and $\b >0$.
Now we consider all $\a > -1$ and $\b > -\a -1/2$, that is all $(\a,\b)$ for which the kernel is defined.
Our aim is to find possibly precise estimates of $K_t^{\a,\b}(x,z)$.
Note that one cannot hope for sharpness for all $\a$ and $\b$ in question since in general the kernel
takes both positive and negative values.

Denote the integral entering the kernel by $\mathcal{I}_{\a,\b}$,
\begin{equation} \label{mIdef}
\mathcal{I}_{\a,\b} = \frac{1}{2^{\a+\b}\Gamma(\a+\b+1)}t^{\a+\b} (xz)^{\a} K_t^{\a,\b}(x,z).
\end{equation}
In order to estimate the kernel, we shall first study $\mathcal{I}_{\a,\b}$.

\subsection{Computation of the triple Bessel function integral $\mathcal{I}_{\a,\b}$}
To compute $\mathcal{I}_{\alpha,\beta}$ we use the formulas \cite[2.12.42 (11)--(13)]{Prudnikov2}, see also
\cite[6.578 (8)]{GR}, expressing it in terms of the associated Legendre functions
(in the corresponding formulas
\cite[Chapter XIII, \S13$\cdot$46(4),(5)]{Wat} and \cite[10.22.72]{handbook} there seems to be an error,
wrong constant in the $Q$ part). What we get splits naturally into the three cases below.

\noindent \textbf{Case 1.} $t < |x-z|$. Then $\mathcal{I}_{\a,\b} \equiv 0$ 
(hence the whole kernel vanishes in this case).

\noindent \textbf{Case 2.} $|x-z| < t < x+z$. Then
$$
\mathcal{I}_{\a,\b} = \frac{1}{\sqrt{2\pi}} \frac{(xz)^{\a+\b-1}}{t^{\a+\b}}
	(\sin v)^{\a+\b-1/2} \mathsf{P}_{\a-1/2}^{1/2-\a-\b}(\cos v),
$$
where $v \in (0,\pi)$ is such that
\begin{equation} \label{def_v}
\cos v = \frac{x^2+z^2-t^2}{2xz} \in (-1,1)
\end{equation}
and $\mathsf{P}$ is the Ferrers function of the first kind 
(associated Legendre function of the first kind on the cut),
cf.\ \cite[Chapter 14]{handbook}.

\noindent \textbf{Case 3.} $t > x+z$. Then
$$
\mathcal{I}_{\a,\b} = \sqrt{\frac{2}{\pi^3}} \frac{(xz)^{\a+\b-1}}{t^{\a+\b}}
	(\sinh u)^{\a+\b-1/2} \sin(\pi \beta) e^{\pi i (\a+\b-1/2)} Q_{\a-1/2}^{1/2-\a-\b}(\cosh u),
$$
where $u > 0$ is such that
\begin{equation} \label{def_u}
\cosh u = \frac{t^2-x^2-z^2}{2xz} \in (1,\infty)
\end{equation}
and $Q$ is the associated Legendre function of the second kind, see \cite[14.3.7]{handbook}.
Since $Q_{\a-1/2}^{1/2-\a-\b}$ is not
defined for $\beta = 1,2,\ldots$, the formula above must be understood in a limiting sense.
To overcome this inconvenience, instead of $Q$ we rather use Olver's function (cf.\ \cite[14.3.10]{handbook})
$$
\mathbf{Q}_{\a-1/2}^{1/2-\a-\b}(y) = e^{\pi i (\a + \b - 1/2)} 
	\frac{Q_{\a-1/2}^{1/2-\a-\b}(y)}{\Gamma(1-\b)},
$$
which (unlike $Q_{\a-1/2}^{1/2-\a-\b}$) is in our situation always real-valued and defined for all $\a$ and $\b$.
This leads to
$$
\mathcal{I}_{\a,\b} = \mathcal{C}(\b) \sqrt{\frac{2}{\pi^3}} \frac{(xz)^{\a+\b-1}}{t^{\a+\b}}
	(\sinh u)^{\a+\b-1/2} \mathbf{Q}_{\a-1/2}^{1/2-\a-\b}(\cosh u),
$$
with $\mathcal{C}(\b) = \sin(\pi \b) \Gamma(1-\b)$ understood in the limiting sense when $\b=1,2,\ldots$. 
Observe that $\mathcal{C}(\b) = 0$ if and only if
$\b = 0,-1,-2,\ldots$ and for such $\b$ this part of the kernel vanishes. Further, $\mathcal{C}(\b)>0$
when $\b >0$ and for $\b < 0$ the sign of $\mathcal{C}(\b)$ is $(-1)^{\lfloor \b \rfloor}$
provided that $\beta \neq -1,-2,\ldots$.
Observe that the constant in question can be written, with a limiting understanding,
as $\mathcal{C}(\beta) = \pi/\Gamma(\beta)$, by Euler's reflection formula
\begin{equation} \label{euler}
\sin(\pi y) \Gamma(1-y) \Gamma(y) = \pi.
\end{equation}

Summing up Cases 1--3, one has
\begin{equation} \label{I_for}
\mathcal{I}_{\a,\b} = \frac{(xz)^{\a+\b-1}}{\sqrt{2\pi}t^{\a+\b}} 
										\begin{cases}
											0, & t < |x-z|, \\												
												(\sin v)^{\a+\b-1/2} \mathsf{P}_{\a-1/2}^{1/2-\a-\b}(\cos v), & |x-z| < t < x+z,\\
												\frac{2}{\Gamma(\b)}
												(\sinh u)^{\a+\b-1/2} \mathbf{Q}_{\a-1/2}^{1/2-\a-\b}(\cosh u), & x+z < t,
										\end{cases}
\end{equation}
where $v$ and $u$ are related to $t,x,z$ by \eqref{def_v} and \eqref{def_u}.

\subsection{Explicit instances of the Legendre functions $\mathsf{P}_{\a-1/2}^{1/2-\a-\b}$
and $\mathbf{Q}_{\a-1/2}^{1/2-\a-\b}$} \label{ss:expl}
For some $\a$ and $\b$ the functions $\mathsf{P}_{\a-1/2}^{1/2-\a-\b}$ and
$\mathbf{Q}_{\a-1/2}^{1/2-\a-\b}$ can be expressed in a more explicit way.
We now derive some of these more elementary expressions. This is of importance for our further development,
since we need to cover certain values of $(\a,\b)$ for which asymptotics of
$\mathsf{P}_{\a-1/2}^{1/2-\a-\b}$ and $\mathbf{Q}_{\a-1/2}^{1/2-\a-\b}$ (see Section \ref{ss:asym} below)
are in a sense singular and seem to be inaccessible in a standard literature on special functions.
Independently, it is of course of interest to know
as explicit form of the kernel as possible, at least for some $\a$ and $\b$.

In what follows we always consider $\mathsf{P}_{\a-1/2}^{1/2-\a-\b}$ on $(-1,1)$ and
$\mathbf{Q}_{\a-1/2}^{1/2-\a-\b}$ on $(1,\infty)$. Also, we always assume $\a > -1$ and
$\a+\b > -1/2$, even though these assumptions can be weakened in some places below.
We will use the formulas (cf.\ \cite[14.3.1,\ 14.3.19,\ 14.3.20]{handbook})
\begin{align} \label{P_form}
\mathsf{P}_{\a-1/2}^{1/2-\a-\b}(y) & = \bigg( \frac{1+y}{1-y} \bigg)^{(1/2-\a-\b)/2}
	\mathbf{F}\Big(\a + \frac{1}{2}, -\a+\frac{1}2; \a+\b+1/2 ; \frac{1-y}2 \Big), \\ \label{Q_form1}
\mathbf{Q}_{\a-1/2}^{1/2-\a-\b}(y) & = 2^{\a-1/2}\Gamma\Big(\a+\frac{1}2\Big)
	\frac{(y+1)^{(1/2-\a-\b)/2}}{(y-1)^{(\a-\b+3/2)/2}}\, \mathbf{F}\Big( \a+\frac{1}2, 1-\b; 2\a+1;
		\frac{2}{1-y}\Big)\\ \label{Q_form2}
& = \frac{-\pi}{2\sin(\pi(\a+\b-1/2))} \\
&	\qquad \times \bigg[ \frac{1}{\Gamma(1-\b)} \bigg(\frac{y-1}{y+1}\bigg)^{(\a+\b-1/2)/2}
	\mathbf{F}\Big( \frac{1}2-\a, \a+\frac{1}2; \a+\b+\frac{1}2; \frac{1-y}2 \Big) \nonumber \\
& \qquad \quad - \frac{1}{\Gamma(2\a+\b)}\Big( \frac{y+1}{y-1} \Big)^{(\a+\b-1/2)/2}
	\mathbf{F}\Big( \frac{1}2-\a, \a+\frac{1}2; \frac{3}2-\a-\b; \frac{1-y}2 \Big) \bigg], \nonumber
\end{align}
where $\mathbf{F}$ stands for Olver's hypergeometric function, see \cite[Sections 15.1,\ 15.2]{handbook},
$$
\mathbf{F}(a,b;c;y) = \frac{1}{\Gamma(c)}\; {_2F_1}(a,b;c;y),
$$
with ${_2F_1}$ being the Gauss hypergeometric function; note that these functions are symmetric in the
first two parameters.
The formula \eqref{Q_form1} has to be understood
in a limiting sense for $\a = -1/2$. In \eqref{Q_form2} we assume that $\a+\b-1/2$ is not integer.
Then in cases when $\b$ is a positive integer or $2\a+\b$ is a non-positive integer, the formula has to
be understood in a limiting sense. 
Furthermore, the following connection with the classical Jacobi polynomials, here denoted by
$\mathbb{P}_m^{\gamma,\delta}$, will be used (cf.\ \cite[18.5.7]{handbook}, \cite[p.\,212]{MOS})
\begin{align} \label{F_jac}
\mathbb{P}_m^{\gamma,\delta}(y) & = \frac{\Gamma(m+\gamma+1)}{m!}\;
	\mathbf{F}\Big(-m,m+\gamma+\delta+1;\gamma+1; \frac{1-y}2\Big) \\
& = \frac{\Gamma(\gamma + \delta + 2m +1)\Gamma(-2m-\gamma-\delta)}{m! \Gamma(\gamma+\delta+m+1)} 
	\Big(\frac{y-1}2\Big)^m \label{F_jac2} \\
& \qquad \times \mathbf{F}\Big( -m, -m - \gamma; -2m-\gamma-\delta; \frac{2}{1-y}\Big), \qquad m=0,1,2,\ldots,
	\nonumber
\end{align}
with suitable limiting understanding of the cases when singularities occur in \eqref{F_jac2}.
For the sake of clarity, we restrict our attention to $-1<y<1$ in \eqref{F_jac} and to $y>1$ in \eqref{F_jac2}.

In the computations of $\mathsf{P}_{\a-1/2}^{1/2-\a-\b}$ and $\mathbf{Q}_{\a-1/2}^{1/2-\a-\b}$ below
we distinguish three main cases.

\noindent \textbf{Case 1.} $\b$ is a non-positive integer, say $\b = -n$, $n=0,1,2,\ldots$.
Observe that in this situation always $\a > -1/2$. We focus on $\mathsf{P}_{\a-1/2}^{1/2-\a-\b}$ only,
since for the considered parameters the form of $\mathbf{Q}_{\a-1/2}^{1/2-\a-\b}$ is irrelevant for
the kernel (the part of the kernel expressed by $\mathbf{Q}_{\a-1/2}^{1/2-\a-\b}$
vanishes for these parameters, see \eqref{I_for}).

Applying the linear transformation (cf.\ \cite[15.8.1]{handbook}) 
$$
\mathbf{F}(a,b;c;y) = (1-y)^{c-a-b}\mathbf{F}(c-a,c-b;c;y)
$$
to \eqref{P_form} and then using \eqref{F_jac} we arrive at an expression for 
$\mathsf{P}_{\a-1/2}^{1/2-\a-\b}$ in terms of Jacobi (actually ultraspherical) polynomials,
\begin{equation} \label{P_expl_1}
\mathsf{P}_{\a-1/2}^{1/2-\a-\b}(y) = \frac{2^{n-\a+1/2} n!}{\Gamma(\a+1/2)}
	\big[ (1+y)(1-y) \big]^{(\a-n-1/2)/2} \mathbb{P}_n^{\a-n-1/2,\a-n-1/2}(y).
\end{equation}
Note that in the special case of $n=0=\beta$ one has $\mathbb{P}_n^{\a-n-1/2,\a-n-1/2}(y) \equiv 1$.

\noindent \textbf{Case 2.} $2\a+\b=0$. Notice that in this case $-1< \a < 1/2$.
Using \eqref{P_form}, \eqref{Q_form1} and the identity (cf.\ \cite[15.4.6]{handbook})
$\mathbf{F}(a,b;b;y) = (1-y)^{-a}/\Gamma(b)$ we find that
\begin{align} \label{P_expl_2}
\mathsf{P}_{\a-1/2}^{1/2-\a-\b}(y) & = \frac{2^{\a+1/2}}{\Gamma(1/2-\a)}
	\big[ (1+y)(1-y) \big]^{-(\a+1/2)/2}, \\ \label{Q_expl_2}
\mathbf{Q}_{\a-1/2}^{1/2-\a-\b}(y) & = \frac{\sqrt{\pi}}{2^{\a+1/2}\Gamma(\a+1)}
	\big[ (y+1)(y-1) \big]^{-(\a+1/2)/2}.
\end{align}
Here, in the computation related to $\mathbf{Q}_{\a-1/2}^{1/2-\a-\b}$, we used the duplication
formula for the gamma function,
\begin{equation} \label{dupli}
\Gamma(y)\Gamma(y+1/2) = \sqrt{\pi}2^{-2y+1}\Gamma(2y).
\end{equation}
Furthermore,
in this computation we treated the value $\a=-1/2$ in a limiting sense.

\noindent \textbf{Case 3.} $\a$ is half of an odd integer, say $\a = n + 1/2$, $n=-1,0,1,2,\ldots$.
Note that in this situation $\b > -n-1$.

To begin with, we first consider the subcase $n=-1$ separately. To this end, $\a = -1/2$ and $\b > 0$.
The formulas we then get are
\begin{align} \label{P_expl_3s}
\mathsf{P}_{\a-1/2}^{1/2-\a-\b}(y) & = \frac{1}{\Gamma(\b)}\Big( \frac{1+y}{1-y} \Big)^{(1-\b)/2}, \\
\mathbf{Q}_{\a-1/2}^{1/2-\a-\b}(y) & = \frac{1}2\bigg[\Big(\frac{y+1}{y-1}\Big)^{(1-\b)/2}
	+ \Big(\frac{y-1}{y+1}\Big)^{(1-\b)/2}\bigg]. 
\label{Q_expl_3s}
\end{align}
To obtain \eqref{P_expl_3s} one observes that the first parameter of $\mathbf{F}$ in \eqref{P_form} vanishes
and, consequently, this function is constant and equal to $1/\Gamma(\b)$.
As for \eqref{Q_expl_3s}, we use \eqref{Q_form1} and, with a limiting understanding,
the duplication formula \eqref{dupli} with $y = \a + 1/2$, to see that
$$
\mathbf{Q}_{\a-1/2}^{1/2-\a-\b}(y) = \Big(\frac{y+1}{y-1}\Big)^{(1-\b)/2}\;
{_2F_1}\Big( \a+\frac{1}{2}, 1-\b; 2\a+1; \frac{2}{1-y} \Big) \bigg|_{\a=-1/2},
$$
where the hypergeometric function is understood in a limiting sense. To find its explicit form
we eliminate the singularity by means of the identity (cf.\ \cite[15.5.15]{handbook})
\begin{align*}
2\,{_2F_1}\Big( \a+\frac{1}{2}, 1-\b; 2\a+1; \frac{2}{1-y} \Big) & = 
	{_2F_1}\Big( \a+\frac{1}{2}, 1-\b; 2\a+2; \frac{2}{1-y} \Big) \\
& 	\qquad + {_2F_1}\Big( \a+\frac{3}{2}, 1-\b; 2\a+2; \frac{2}{1-y} \Big).
\end{align*}
Letting $\a = -1/2$ on the right-hand side here and using the formula ${_2F_1}(a,b;a;y) = (1-y)^{-b}$
(essentially stated in Case 2 in terms of $\mathbf{F}$) we see that
$$
{_2F_1}\Big( \a+\frac{1}{2}, 1-\b; 2\a+1; \frac{2}{1-y} \Big) \bigg|_{\a=-1/2} =
	\frac{1}2 \bigg[ 1 + \Big(\frac{y+1}{y-1}\Big)^{\b-1}\bigg].
$$
Now \eqref{Q_expl_3s} follows.

From now, to the end of Case 3, we assume $n \ge 0$.
By \eqref{P_form} and \eqref{F_jac} (recall that $\mathbf{F}$ is symmetric in the first two parameters)
one easily gets
\begin{equation} \label{P_expl_3}
\mathsf{P}_{\a-1/2}^{1/2-\a-\b}(y) = \frac{n!}{\Gamma(2n+\b+1)} \Big(\frac{1+y}{1-y}\Big)^{-(n+\b)/2}
	\mathbb{P}_n^{n+\b,-n-\b}(y).
\end{equation}
Finally, to compute $\mathbf{Q}_{\a-1/2}^{1/2-\a-\b}$ for $\a=n+1/2$, we first assume that
$\b$ is not integer and employ \eqref{Q_form2}. Then, with the aid of \eqref{F_jac},
\begin{align} \label{Q_expl_3}
\mathbf{Q}_{\a-1/2}^{1/2-\a-\b}(y) & = \frac{\pi n!}{2\sin(\pi(n+\b))\Gamma(1-\b)\Gamma(2n+\b+1)} \\
& \quad \times	\bigg[ \Big(\frac{y+1}{y-1}\Big)^{(n+\b)/2} \mathbb{P}_n^{-n-\b,n+\b}(y)
		- \Big(\frac{y-1}{y+1}\Big)^{(n+\b)/2} \mathbb{P}_n^{n+\b,-n-\b}(y) \bigg]. \nonumber
\end{align}
It remains to consider $\b = m$, $m=1,2,\ldots$ (the case when $\b$ is a non-positive integer is irrelevant for
our purposes). In this situation $\mathbf{Q}_{\a-1/2}^{1/2-\a-\b}$ can be computed in fact for any
$\a$ (recall that at the moment we are considering $\a \neq -1/2$). Indeed, using \eqref{Q_form1} and
\eqref{F_jac2} we arrive at
\begin{align} \label{Q_expl_3ss}
\mathbf{Q}_{\a-1/2}^{1/2-\a-\b}(y) & =
	\frac{\sqrt{\pi}(-1)^{m+1}(m-1)!\, \Gamma(2\a+1)}{2^{\a-m+3/2}\Gamma(\a+1)\Gamma(2\a+m)} \\
& \qquad \times \big[ (y+1)(y-1) \big]^{(1/2-\a-m)/2} \mathbb{P}_{m-1}^{1/2-\a-m,1/2-\a-m}(y). \nonumber
\end{align}
To be precise, here we also used the duplication and reflection formulas \eqref{dupli} and \eqref{euler}.
Note that for $\b=1$ and $\a=-1/2$ one has $\mathbf{Q}_{\a-1/2}^{1/2-\a-\b}(y) \equiv 1$.

\begin{remark} \label{rem:ker_expl}
From the above considerations it follows that the kernel $K_t^{\a,\b}(x,z)$ can be written explicitly, by
means of elementary functions and possibly Jacobi polynomials,
whenever $2\a+\b=0$ or $-\b \in \mathbb{N}$ or $\a+1/2 \in \mathbb{N}$; see Figure \ref{fig:Kexpl}
where bold straight lines represent the cases when the kernel has the explicit form.
\end{remark}

\begin{figure}
\centering
\begin{tikzpicture}[scale=5]
\draw[very thick,dashed] (2.1,0.5) -- (3.4,-0.8);
\draw[line width=2pt] (2.3,0.9) -- (2.9,-0.3);
\draw (2.0,0.3) node {$\alpha+\beta=-1/2$};
\draw (3.16,-0.87) node {$2\alpha+\beta=0$};
\draw[very thick, dashed] (2.3,1.1) -- (2.3,0.3);
\draw[very thick, dashed] (2.3,0.9) -- (3.15,-0.8);
\draw[line width=2pt] (2.5,0.1) -- (2.5,1.1);
\draw[line width=2pt] (2.5,0.1) -- (3.9,0.1);
\draw[line width=2pt] (2.9,-0.3) -- (3.9,-0.3);
\draw[line width=2pt] (3.3,-0.7) -- (3.9,-0.7);
\draw[line width=2pt] (2.9,1.1) -- (2.9,-0.3);
\draw[line width=2pt] (3.3,1.1) -- (3.3,-0.7);
\draw[line width=2pt] (3.7,1.1) -- (3.7,-0.8);
\filldraw[fill=white] (2.3,0.9) circle(0.6pt);
\filldraw[fill=white] (2.9,-0.3) circle(0.6pt);
\filldraw[fill=white] (2.5,0.1) circle(0.6pt);
\filldraw[fill=white] (3.3,-0.7) circle (0.6pt);
\draw[arrows=-angle 60] (2.7,-0.8) -- (2.7,1.15);
\draw[thin] (2,0.1) -- (2.48,0.1);
\draw[arrows=-angle 60] (2.52,0.1) -- (3.95,0.1);
\draw[very thin] (2.3,0.1) -- (2.3,0.13);
\draw[very thin] (3.1,0.1) -- (3.1,0.13);
\draw[very thin] (3.5,0.1) -- (3.5,0.13);
\draw[very thin] (2.7,0.9) -- (2.73,0.9);
\draw[very thin] (2.7,0.5) -- (2.73,0.5);
\draw[very thin] (2.7,-0.3) -- (2.73,-0.3);
\draw[very thin] (2.7,-0.7) -- (2.73,-0.7);
\node at (2.3,0) {$-1$};
\node at (3.1,0) {$1$};
\node at (3.5,0) {$2$};
\node at (2.8,-0.3) {$-1$};
\node at (2.8,-0.7) {$-2$};
\node at (2.8,0.5) {$1$};
\node at (2.8,0.9) {$2$};
\node at (4,0.1) {$\alpha$};
\node at (2.7,1.2) {$\beta$};
\end{tikzpicture}
\caption{Cases when the kernel has explicit form} \label{fig:Kexpl}
\end{figure}
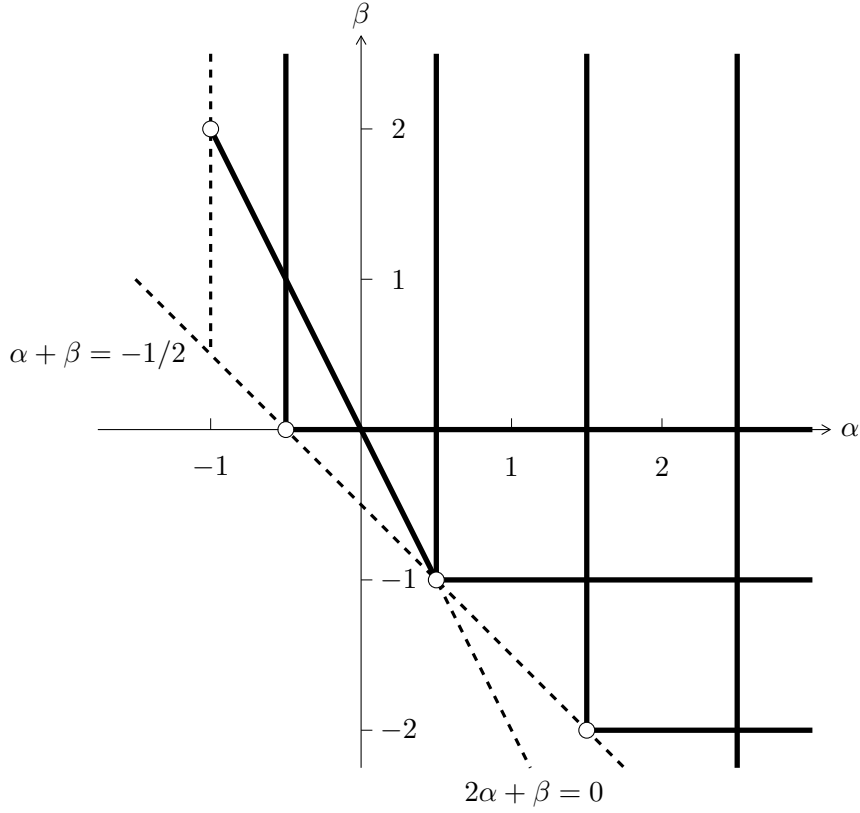

\subsection{Asymptotics of $\mathsf{P}_{\a-1/2}^{1/2-\a-\b}$ and $\mathbf{Q}_{\a-1/2}^{1/2-\a-\b}$}
\label{ss:asym}
To estimate the kernel our strategy will be to employ known asymptotics of the Legendre functions
near singular points ($-1^+$ and $1^-$ in case of $\mathsf{P}_{\a-1/2}^{1/2-\a-\b}$,
and $1^+$ and $\infty$ in case of $\mathbf{Q}_{\a-1/2}^{1/2-\a-\b}$).
These asymptotics will be taken either from the literature, cf.\ \cite[Section 3.9.2]{EMOT},
\cite[Table 4.8.2]{MOS}, \cite[Section 14.8]{handbook},
or from the explicit formulas derived in the previous section in case $(\a,\b)$ belongs to
one of the exceptional sets $E^{\mathsf{P}}$, $E^{\mathbf{Q}}$ that will be defined in a moment.
In the asymptotic expressions below we always write multiplicative constants,
usually depending on $\a$ and $\b$, when they may decide about signs.

Recall that we are considering $\a > -1$ and $\b > -\a-1/2$.
The cases of singular points $1^-$ and $\infty$ are clear, we have
\begin{align} \label{as_P_reg}
\mathsf{P}_{\a-1/2}^{1/2-\a-\b}(y) & \simeq (1-y)^{-(1/2-\a-\b)/2}, \qquad y \to 1^-, \\
\mathbf{Q}_{\a-1/2}^{1/2-\a-\b}(y) & \simeq y^{-\a-1/2}, \qquad y \to \infty. \label{as_Q_reg}
\end{align}

In order to treat the remaining two singular points, we define the exceptional sets
\begin{align*}
E^{\mathsf{P}} & = \big\{ (\a,\b) : \a+\b \ge 1/2 \;\; \textrm{and} \;\; 
		[-\b \in \mathbb{N} \;\; \textrm{or} \;\; 2\a+\b = 0] \big\} \\ 
& \qquad \cup \big\{ (\a,\b) : \a+\b < 1/2 \;\; \textrm{and} \;\; \a+1/2 \in \mathbb{N} \big\},\\
E^{\mathbf{Q}} & = \big\{ (\a,\b) : \a + \b \ge 1/2 \;\; \textrm{and} \;\;
	2\a+\b = 0 \big\} \cup \big\{ (\a,\b) : \a+\b < 1/2 \;\; \textrm{and} \;\; \b=1 \big\}.
\end{align*}
These sets are visualized by Figures \ref{fig:EP} and \ref{fig:EQ} with bold straight lines and black dots.

\begin{figure}
\centering
\begin{tikzpicture}[scale=4]
\draw[very thick,dashed] (2.1,0.5) -- (3.4,-0.8);
\draw[very thick,dashed] (2.1,0.9) -- (3.8,-0.8);
\draw[line width=2pt] (2.3,0.9) -- (2.5,0.5);
\draw (1.8,0.4) node {$\alpha+\beta=-1/2$};
\draw (1.8,0.8) node {$\alpha+\beta=1/2$};
\draw (2.96,-0.87) node {$2\alpha+\beta=0$};
\draw[very thick, dashed] (2.3,1.1) -- (2.3,0.3);
\draw[very thick, dashed] (2.3,0.9) -- (3.15,-0.8);
\draw[line width=2pt] (2.9,0.1) -- (3.9,0.1);
\draw[line width=2pt] (3.3,-0.3) -- (3.9,-0.3);
\draw[line width=2pt] (3.7,-0.7) -- (3.9,-0.7);
\draw[line width=2pt] (2.5,0.5) -- (2.5,0.1);
\draw[line width=2pt] (2.9,0.1) -- (2.9,-0.3);
\draw[line width=2pt] (3.3,-0.3) -- (3.3,-0.7);
\draw[line width=2pt] (3.7,-0.7) -- (3.7,-0.8);
\filldraw[fill=black] (2.5,0.5) circle(0.6pt);
\filldraw[fill=black] (3.7,-0.7) circle (0.6pt);
\filldraw[fill=black] (3.3,-0.3) circle (0.6pt);
\filldraw[fill=black] (2.9,0.1) circle (0.6pt);
\filldraw[fill=white] (2.3,0.9) circle (0.6pt);
\filldraw[fill=white] (2.5,0.1) circle (0.6pt);
\filldraw[fill=white] (2.9,-0.3) circle (0.6pt);
\filldraw[fill=white] (3.3,-0.7) circle (0.6pt);
\draw[arrows=-angle 60] (2.7,-0.8) -- (2.7,1.15);
\draw[thin] (2,0.1) -- (2.48,0.1);
\draw[arrows=-angle 60] (2.52,0.1) -- (3.95,0.1);
\draw[very thin] (2.3,0.1) -- (2.3,0.13);
\draw[very thin] (3.1,0.1) -- (3.1,0.13);
\draw[very thin] (3.5,0.1) -- (3.5,0.13);
\draw[very thin] (2.7,0.9) -- (2.73,0.9);
\draw[very thin] (2.7,0.5) -- (2.73,0.5);
\draw[very thin] (2.7,-0.3) -- (2.73,-0.3);
\draw[very thin] (2.7,-0.7) -- (2.73,-0.7);
\node at (2.3,0) {$-1$};
\node at (3.1,0) {$1$};
\node at (3.5,0) {$2$};
\node at (2.8,-0.3) {$-1$};
\node at (2.8,-0.7) {$-2$};
\node at (2.8,0.5) {$1$};
\node at (2.8,0.9) {$2$};
\node at (4,0.1) {$\alpha$};
\node at (2.7,1.2) {$\beta$};
\end{tikzpicture}
\caption{The exceptional set $E^{\mathsf{P}}$} \label{fig:EP}
\end{figure}

\begin{figure}
\centering
\begin{tikzpicture}[scale=4]
\draw[very thick,dashed] (4.6,0.5) -- (5.4,-0.3);
\draw[very thick,dashed] (4.6,0.9) -- (5.8,-0.3);
\draw[line width=2pt] (4.8,0.9) -- (5,0.5);
\draw[line width=2pt] (4.8,0.5) -- (5,0.5);
\draw (4.3,0.4) node {$\alpha+\beta=-1/2$};
\draw (4.3,0.8) node {$\alpha+\beta=1/2$};
\draw (5.46,-0.57) node {$2\alpha+\beta=0$};
\draw[very thick, dashed] (4.8,1.1) -- (4.8,0.3);
\draw[very thick, dashed] (4.8,0.9) -- (5.5,-0.5);
\filldraw[fill=black] (5,0.5) circle(0.6pt);
\filldraw[fill=white] (4.8,0.9) circle(0.6pt);
\filldraw[fill=white] (4.8,0.5) circle(0.6pt);
\draw[arrows=-angle 60] (5.2,-0.5) -- (5.2,1.15);
\draw[arrows=-angle 60] (4.5,0.1) -- (5.85,0.1);
\draw[very thin] (4.8,0.1) -- (4.8,0.13);
\draw[very thin] (5.6,0.1) -- (5.6,0.13);
\draw[very thin] (5.2,0.9) -- (5.23,0.9);
\draw[very thin] (5.2,0.5) -- (5.23,0.5);
\draw[very thin] (5.2,-0.3) -- (5.23,-0.3); 
\node at (4.8,0) {$-1$};
\node at (5.6,0) {$1$};
\node at (5.3,-0.3) {$-1$};
\node at (5.3,0.5) {$1$};
\node at (5.3,0.9) {$2$};
\node at (5.9,0.1) {$\alpha$};
\node at (5.2,1.2) {$\beta$};
\end{tikzpicture}
\caption{The exceptional set $E^{\mathbf{Q}}$} \label{fig:EQ}
\end{figure}

At $-1^+$ it happens that
\begin{equation} \label{as_P}
\mathsf{P}_{\a-1/2}^{1/2-\a-\b}(y) \simeq
	\begin{cases}
		(1+y)^{(1/2-\a-\b)/2} \frac{1}{\Gamma(2\a+\b)\Gamma(\b)}, & \qquad \a+\b > 1/2,\\
		-\log(1+y) \sin[\pi(1/2-\a)], & \qquad \a+\b=1/2,\\
		(1+y)^{-(1/2-\a-\b)/2} \sin[\pi(1/2-\a)], & \qquad \a+\b< 1/2,
	\end{cases}
\end{equation}
as $y \to -1^+$, provided that $(\a,\b) \notin E^{\mathsf{P}}$.
On the other hand, for $(\a,\b) \in E^{\mathsf{P}}$ the asymptotic is different
(or rather inverse; moreover, there are no logarithms when $\a+\b=1/2$). More precisely,
for $(\a,\b) \in E^{\mathsf{P}}$ we have
$$
\mathsf{P}_{\a-1/2}^{1/2-\a-\b}(y) \simeq
	\begin{cases}
		(1+y)^{(\a+\b-1/2)/2} c_1(\a,\b), & \qquad \a+\b \ge 1/2,\\
		(1+y)^{(1/2-\a-\b)/2} c_2(\a,\b), & \qquad \a+\b< 1/2.
	\end{cases}
$$
Here the constants $c_1,c_2$ are to indicate signs, $c_1 = 1$ if $2\a+\b=0$,
$c_1 = (-1)^{\b}$ if $\b=0,-1,-2,\ldots$, $c_2 = 1$ if $\a=-1/2$ and $c_2=(-1)^{\a-1/2}$ for
$\a = 1/2,3/2,5/2,\ldots$.

Finally, as $y \to 1^+$ and $(\a,\b) \notin E^{\mathbf{Q}}$, we have
\begin{equation} \label{as_Q}
\mathbf{Q}_{\a-1/2}^{1/2-\a-\b}(y) \simeq
	\begin{cases}
		(y-1)^{(1/2-\a-\b)/2}\frac{1}{\Gamma(2\a+\b)}, & \qquad \a+\b >1/2, \\
		-\log(y-1)\frac{1}{\Gamma(\a+1/2)}, & \qquad \a+\b = 1/2, \\
		(y-1)^{(\a+\b-1/2)/2}\frac{1}{\Gamma(1-\b)}, & \qquad \a+\b < 1/2,
	\end{cases}
\end{equation}
whereas for $(\a,\b) \in E^{\mathbf{Q}}$ (see also \eqref{Q_expl_3ss} with $m=1=\b$)
$$
\mathbf{Q}_{\a-1/2}^{1/2-\a-\b}(y) \simeq
	\begin{cases}
		(y-1)^{(\a+\b-1/2)/2}, & \qquad \a+\b \ge 1/2, \\
		(y-1)^{(1/2-\a-\b)/2}, & \qquad \a+\b < 1/2.
	\end{cases}
$$

Neglecting signs, the asymptotics at $-1^+$ and $1^+$ can be written in a more compact way, respectively
$$
\big|\mathsf{P}_{\a-1/2}^{1/2-\a-\b}(y)\big| \simeq
	\begin{cases}
		(1+y)^{-|\a+\b-1/2|/2}, & (\a,\b) \notin E^{\mathsf{P}}, \; \a+\b \neq 1/2, \\
		- \log(1+y), & (\a,\b) \notin E^{\mathsf{P}}, \; \a+\b = 1/2, \\
		(1+y)^{|\a+\b-1/2|/2}, & (\a,\b) \in E^{\mathsf{P}},
	\end{cases}
$$
and
$$
\big| \mathbf{Q}_{\a-1/2}^{1/2-\a-\b}(y) \big| \simeq
	\begin{cases}
		(y-1)^{-|\a+\b-1/2|/2}, & (\a,\b) \notin E^{\mathbf{Q}}, \; \a+\b \neq 1/2, \\
		- \log(y-1), & (\a,\b) \notin E^{\mathbf{Q}}, \; \a+\b = 1/2, \\
		(y-1)^{|\a+\b-1/2|/2}, & (\a,\b) \in E^{\mathbf{Q}}.
	\end{cases}
$$

\subsection{Zeros of $\mathsf{P}_{\a-1/2}^{1/2-\a-\b}$ and $\mathbf{Q}_{\a-1/2}^{1/2-\a-\b}$}
\label{ss:zero}
We keep considering these functions on $(-1,1)$ and $(1,\infty)$, respectively.
Since we are going to obtain possibly sharp estimates in terms of asymptotics of these functions,
it is important to know if, given $\a$ and $\b$, they have zeros. If this is not the case,
asymptotics invoked in Section \ref{ss:asym}, see \eqref{as_P_reg} and \eqref{as_Q_reg},
imply that the functions are strictly positive.

We will prove the following. 
\begin{propo} \label{prop:zeros}
Let $\a > -1$ and $\a + \b > -1/2$.
\begin{itemize}
\item[(a)]
The function $\mathsf{P}_{\a-1/2}^{1/2-\a-\b}$ has no zeros in $(-1,1)$ if $\a \ge -1/2$ and $\b \ge 0$
or $\a < -1/2$ and $\b > -2\a$ or $\a \le 1/2$ and $\b < 0$. Otherwise, $\mathsf{P}_{\a-1/2}^{1/2-\a-\b}$
has at least one zero in $(-1,1)$.
\item[(b)]
The function $\mathbf{Q}_{\a-1/2}^{1/2-\a-\b}$ has exactly one zero if $\a < -1/2$ and $1 < \b < -2\a$.
Otherwise, $\mathbf{Q}_{\a-1/2}^{1/2-\a-\b}$ has no zeros in $(1,\infty)$.
\end{itemize}
\end{propo}
See Figures \ref{fig:zeroP} and \ref{fig:zeroQ} where gray regions together
with bold straight lines and black dots represent pairs $(\a,\b)$ for which $\mathsf{P}_{\a-1/2}^{1/2-\a-\b}$
has no zeros in $(-1,1)$ and $\mathbf{Q}_{\a-1/2}^{1/2-\a-\b}$ has no zeros in $(1,\infty)$, respectively.

\begin{figure}
\centering
\begin{tikzpicture}[scale=5]
\fill[black!10!white]
(2.3,1.1) -- (2.3,0.9) -- (2.5,0.5) -- (2.5,0.1) -- (2.9,-0.3)-- (2.9,0.1) --(3.7,0.1) -- (3.7,1.1)  -- cycle;
\draw[ultra thick] (2.5,0.48) -- (2.5,0.1);
\draw[ultra thick] (2.9,0.1) -- (3.7,0.1);
\draw (3.13,0.55) node {No zeros};
\draw[very thick,dashed] (2.1,0.5) -- (3.1,-0.5);
\draw (3.23,-0.35) node {$\alpha+\beta=-1/2$};
\draw[very thick, dashed] (2.3,0.9) -- (2.49,0.52);
\draw[very thick, dashed] (2.3,1.1) -- (2.3,0.3);
\draw[ultra thick] (2.9,0.1) -- (2.9,-0.3);
\filldraw[fill=black] (2.5,0.5) circle(0.6pt);
\filldraw[fill=white] (2.5,0.1) circle (0.6pt);
\filldraw[fill=white] (2.9,-0.3) circle (0.6pt);
\filldraw[fill=black] (2.9,0.1) circle(0.6pt);
\draw[arrows=-angle 60] (2.7,-0.5) -- (2.7,1.15);
\draw[very thin] (2,0.1) -- (2.48,0.1);
\draw[arrows=-angle 60] (2.52,0.1) -- (3.75,0.1);
\draw[very thin] (2.3,0.1) -- (2.3,0.13);
\draw[very thin] (3.1,0.1) -- (3.1,0.13);
\draw[very thin] (3.5,0.1) -- (3.5,0.13);
\draw[very thin] (2.7,0.9) -- (2.73,0.9);
\draw[very thin] (2.7,0.5) -- (2.73,0.5);
\draw[very thin] (2.7,-0.3) -- (2.73,-0.3);
\node at (2.3,0) {$-1$};
\node at (2.75,0.05) {$0$};
\node at (3.1,0) {$1$};
\node at (3.5,0) {$2$};
\node at (2.8,-0.3) {$-1$};
\node at (2.8,0.5) {$1$};
\node at (2.8,0.9) {$2$};
\node at (3.8,0.1) {$\alpha$};
\node at (2.7,1.2) {$\beta$};
\end{tikzpicture}
\caption{Zeros of $\mathrm{P}_{\alpha-1/2}^{1/2-\alpha-\beta}$ in $(-1,1)$} \label{fig:zeroP}
\end{figure}
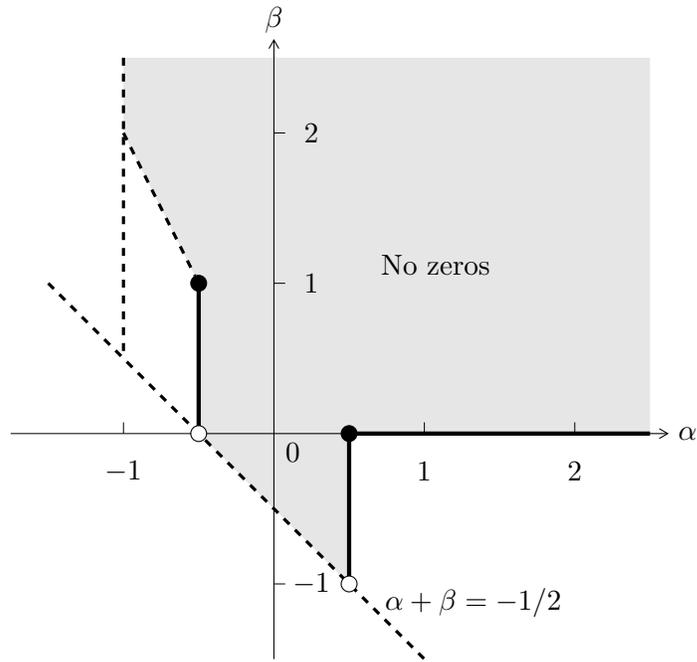

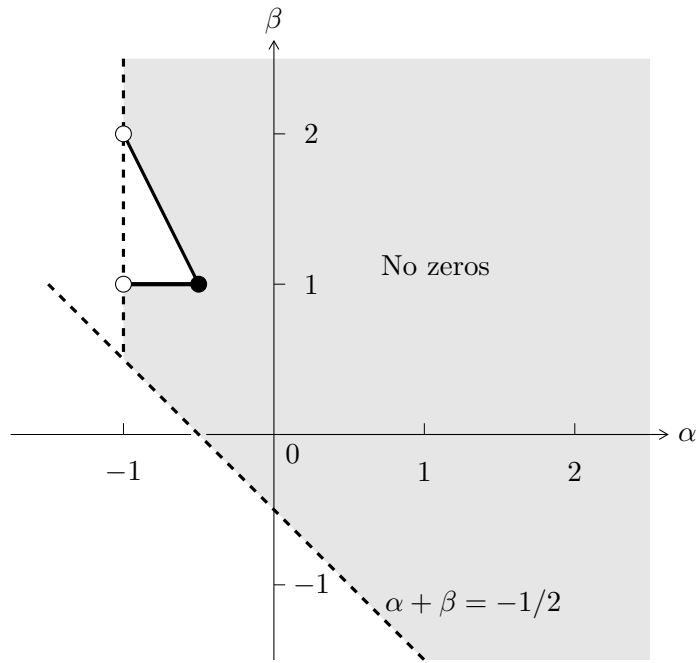
\begin{figure}
\centering
\begin{tikzpicture}[scale=5]
\fill[black!10!white]
(2.3,1.1) -- (2.3,0.9) -- (2.5,0.5) -- (2.3,0.5)  -- (2.3,0.3)--  
(3.1,-0.5)--  (3.7,-0.5)-- (3.7,1.1)  -- cycle;
\draw[ultra thick] (2.5,0.5) -- (2.3,0.5);
\draw (3.13,0.55) node {No zeros};
\draw[very thick,dashed] (2.1,0.5) -- (3.1,-0.5);
\draw (3.23,-0.35) node {$\alpha+\beta=-1/2$};
\draw[very thick] (2.3,0.9) -- (2.5,0.5);
\draw[very thick, dashed] (2.3,1.1) -- (2.3,0.3);
\draw[arrows=-angle 60] (2.7,-0.5) -- (2.7,1.15);
\draw[very thin] (2,0.1) -- (2.48,0.1);
\draw[arrows=-angle 60] (2.52,0.1) -- (3.75,0.1);
\draw[very thin] (2.3,0.1) -- (2.3,0.13);
\draw[very thin] (3.1,0.1) -- (3.1,0.13);
\draw[very thin] (3.5,0.1) -- (3.5,0.13);
\draw[very thin] (2.7,0.9) -- (2.73,0.9);
\draw[very thin] (2.7,0.5) -- (2.73,0.5);
\draw[very thin] (2.7,-0.3) -- (2.73,-0.3);
\filldraw[fill=black] (2.5,0.5) circle(0.6pt);
\filldraw[fill=white] (2.3,0.9) circle(0.6pt);
\filldraw[fill=white] (2.3,0.5) circle(0.6pt);
\node at (2.3,0) {$-1$};
\node at (2.75,0.05) {$0$};
\node at (3.1,0) {$1$};
\node at (3.5,0) {$2$};
\node at (2.8,-0.3) {$-1$};
\node at (2.8,0.5) {$1$};
\node at (2.8,0.9) {$2$};
\node at (3.8,0.1) {$\alpha$};
\node at (2.7,1.2) {$\beta$};
\end{tikzpicture}
\caption{Zeros of $\mathbf{Q}_{\alpha-1/2}^{1/2-\alpha-\beta}$ in $(1,\infty)$} \label{fig:zeroQ}
\end{figure}

\begin{proof}[{Proof of Proposition \ref{prop:zeros}}]
We first prove (a). From the explicit formulas \eqref{P_expl_1}, \eqref{P_expl_2}, \eqref{P_expl_3s}
and \eqref{P_expl_3} we know, respectively, that there is no zero when $\b =0$ or $\b = -2\a$
or $\a = -1/2$ or [$\a=1/2$ and $\b < 0$]. Further, by the first item (a) in
\cite[Section 14.16(ii)]{handbook},
it follows that there is no zero if $\a,\b > 0$ and $\a+\b \ge 1/2$. Combining the same condition
with \cite[14.9.5]{handbook} we infer that no zeros may occur if $\a < 0$ and $\b > -2\a$ and
$\a+\b \ge 1/2$. The lack of zeros in case $-1/2< \a < 1/2$ and $\a+\b < 1/2$ follows from
item (c) in \cite[Section 14.16(ii)]{handbook}, combined with \cite[14.9.5]{handbook} when $\a < 0$ comes
into play. Altogether, the above shows that there are no zeros for $(\a,\b)$ indicated in item (a)
of the proposition.

We now treat $(\a,\b)$ for which there is at least one zero. The explicit formula \eqref{P_expl_3}
shows that this is the case when $\a=1/2,3/2,5/2,\ldots$ and $\b < 1/2-\a$ (here we use the standard
fact that Jacobi polynomials $\mathbb{P}_m^{\gamma,\delta}$ have $m$ zeros in the interval $(-1,1)$
whenever $\gamma,\delta>-1$). Further, according to the first item (a) in \cite[Section 14.16(ii)]{handbook},
there is at least one zero provided that $\a+\b \ge 1/2$ and $\b < 0$. The regions defined by the conditions
$n+1/2 < \a < n+3/2$, $\a+\b < 1/2$, $2\a+\b> n+1$, $n \in \mathbb{N}$, are covered by the first item
(b) in \cite[Section 14.16(ii)]{handbook}, while the adjacent regions $n+1/2 < \a < n+3/2$, $2\a+\b \le n+1$,
$n \ge 1$, by the second item (a) in \cite[Section 14.16(ii)]{handbook}. The remaining two regions defined by
$\a < -1/2$, $\b < -2\a$ and $1/2 < \a < 3/2$, $2\a+\b \le 1$, respectively, can be dealt with the aid
of the asymptotics from Section \ref{ss:asym}. Indeed, taking into account \eqref{as_P_reg}, it is
enough, for continuity reason, to check that in the regions in question the asymptotic expressions at
$-1^+$ are negative. But this immediately follows from \eqref{as_P}.

Passing to (b), according to \cite[Section 14.16(iii)]{handbook}, $\mathbf{Q}_{\a-1/2}^{1/2-\a-\b}$
has no zeros in $(1,\infty)$ when $\a > -1/2$ and at most one zero when $\a < -1/2$.
Moreover, by the explicit formulas \eqref{Q_expl_3s}, \eqref{Q_expl_2} and \eqref{Q_expl_3ss}
we know, respectively, that no zeros occur when $\a = -1/2$ or $\a < -1/2$ and $\b = -2\a$ or
$\a < -1/2$ and $\b=1$. Further, the lack of zeros for $\a < -1/2$ and [$\b < 1$ or $\b > -2\a$]
follows by applying Whipple's formula \cite[14.9.16]{handbook} and then using the criterion from
\cite[Section 14.16(iii)]{handbook} for the associated Legendre function of the first kind.
More precisely, this covers $\b > -2\a$ and to treat $\b < 1$ one combines the criterion just
mentioned with the formula \cite[14.9.11]{handbook}.
Finally, there is one zero in the triangle $\a < -1/2$, $1 < \b < -2\a$,
in view of the asymptotic expressions \eqref{as_Q_reg} and \eqref{as_Q}. 
The conclusion follows.
\end{proof}

It is worth observing that the fact that $\mathsf{P}_{\a-1/2}^{1/2-\a-\b}$ and
$\mathbf{Q}_{\a-1/2}^{1/2-\a-\b}$ have no zeros in $(-1,1)$ and $(1,\infty)$, respectively,
when $\a > -1/2$ and $\b > 0$, is a straightforward consequence of \eqref{ifK} and \eqref{I_for}.

\subsection{Estimates of the kernel $K_t^{\a,\b}(x,z)$}
Using \eqref{mIdef}, \eqref{I_for} and
the asymptotics from Section \ref{ss:asym} and taking into account continuity of the functions
under consideration, we conclude the following estimates for the kernel.

\noindent \textbf{Region} $|x-z| < t < x+z$. If $(\a,\b) \notin E^{\mathsf{P}}$, then
\begin{align*}
& |K_t^{\a,\b}(x,z)| \\ 
& \quad \lesssim \frac{(xz)^{\b-1}}{t^{2\a+2\b}} \big[ (\sin v)^2 (1-\cos v)\big]^{(\a+\b-1/2)/2}
	\begin{cases}
		(1+\cos v)^{-(\a+\b-1/2)/2}, & \a + \b > 1/2, \\
		1-\log(1+\cos v), & \a + \b =1/2, \\
		(1+ \cos v)^{(\a+\b-1/2)/2}, & \a + \b < 1/2,
	\end{cases}
\end{align*}
whereas in case $(\a,\b) \in E^{\mathsf{P}}$ one has a different and simpler (no logarithmic case) bound
\begin{align*}
& |K_t^{\a,\b}(x,z)| \\
& \quad \lesssim \frac{(xz)^{\b-1}}{t^{2\a+2\b}} \big[ (\sin v)^2 (1-\cos v)\big]^{(\a+\b-1/2)/2}
	\begin{cases}
		(1+\cos v)^{(\a+\b-1/2)/2}, & \a + \b \ge 1/2, \\
		(1+ \cos v)^{-(\a+\b-1/2)/2}, & \a + \b < 1/2.
	\end{cases}
\end{align*}
Here $v$ is determined by \eqref{def_v}. For those $(\a,\b)$ for which $\mathsf{P}_{\a-1/2}^{1/2-\a-\b}$
has no zeros in $(-1,1)$ (see Proposition \ref{prop:zeros}) the estimates are sharp, one can replace
$\lesssim$ by $\simeq$ and, moreover, suppress the absolute value of the kernel.
Otherwise, the estimates are sharp provided that $\cos v$ is in a (sufficiently small) neighborhood
of $-1$ or $1$.

\noindent \textbf{Region} $x+z < t$. If $(\a,\b) \notin E^{\mathbf{Q}}$
and in addition $-\b \notin \mathbb{N}$, then
\begin{align*}
& |(-1)^{\lfloor \b \wedge 0 \rfloor}K_t^{\a,\b}(x,z)| \\ 
& \quad \lesssim \frac{(xz)^{\b-1}}{t^{2\a+2\b}} (\sinh u)^{\a+\b-1/2} (\cosh u +1)^{-\a-1/2}
	\begin{cases}
		\left({\frac{\cosh u -1}{\cosh u +1}}\right)^{-(\a+\b-1/2)/2}, & \a + \b > 1/2, \\
		1-\log\left({\frac{\cosh u -1}{\cosh u +1}}\right), & \a + \b =1/2, \\
		\left({\frac{\cosh u -1}{\cosh u +1}}\right)^{(\a+\b-1/2)/2}, & \a + \b < 1/2,
	\end{cases}
\end{align*}
and in case $(\a,\b) \in E^{\mathbf{Q}}$ we have
\begin{align*}
& |K_t^{\a,\b}(x,z)| \\ 
& \quad \lesssim \frac{(xz)^{\b-1}}{t^{2\a+2\b}} (\sinh u)^{\a+\b-1/2} (\cosh u +1)^{-\a-1/2}
	\begin{cases}
		\left({\frac{\cosh u -1}{\cosh u +1}}\right)^{(\a+\b-1/2)/2}, & \a + \b \ge 1/2, \\
		\left({\frac{\cosh u -1}{\cosh u +1}}\right)^{-(\a+\b-1/2)/2}, & \a + \b < 1/2.
	\end{cases}
\end{align*}
Here $u$ is determined by \eqref{def_u}. For those $(\a,\b)$ for which $\mathbf{Q}_{\a-1/2}^{1/2-\a-\b}$
has no zeros in $(1,\infty)$ (see Proposition \ref{prop:zeros}) the estimates are sharp, one can replace
$\lesssim$ by $\simeq$ and, moreover, suppress the absolute values.
Otherwise, the estimates are sharp provided that $\cosh v$ is sufficiently large, or sufficiently close
to $1$.

From the above bounds we can readily get estimates of $K_t^{\a,\b}(x,z)$ in terms of $t,x,z$.
To do that, we use the following identities that hold when $|x-z| < t$:
\begin{align*}
1+ \cos v  = \cosh u -1 & = \frac{1}2 \frac{|(x+z)^2-t^2|}{xz}, \\
1- \cos v  = \cosh u +1 & = \frac{1}2 \frac{t^2-(x-z)^2}{xz}, \\
\sin^2 v  = \sinh^2 u & = \frac{1}4 \frac{|(x+z)^2-t^2|}{xz} \frac{t^2-(x-z)^2}{xz}.
\end{align*}
As the outcome, taking also into account \eqref{st5}, we obtain the main result concerning
pointwise estimates of $K_t^{\a,\b}(x,z)$.

\begin{thm} \label{thm:kerest2}
Assume that $\a > -1$ and $\a+\b > -1/2$. Let $t,x,z > 0$.
\begin{itemize}
\item[(1)]
	The kernel $K_t^{\a,\b}(x,z)$ vanishes when $t < |x-z|$.
\item[(2)]
	The following estimates hold uniformly in $|x-z| < t < x+z$.
	\begin{itemize}
		\item[(2a)]
			If $-\b \in \mathbb{N}$ or $2\a+\b = 0$, then
			$$
				|K_t^{\a,\b}(x,z)| \lesssim 
				\frac{(xz)^{-\a-1/2}}{t^{2\a+2\b}} \big[ t^2-(x-z)^2\big]^{\a+\b-1/2}
				\bigg( \frac{(x+z)^2-t^2}{xz}\bigg)^{\a+\b-1/2}.
			$$
		\item[(2b)]
			If $\a+1/2 \in \mathbb{N}$, then
			$$
				|K_t^{\a,\b}(x,z)| \lesssim 
				\frac{(xz)^{-\a-1/2}}{t^{2\a+2\b}} \big[ t^2-(x-z)^2\big]^{\a+\b-1/2}.
			$$
		\item[(2c)]
			For all $(\a,\b)$ not covered by items (2a) and (2b),
			\begin{align*}
				|K_t^{\a,\b}(x,z)| & \lesssim
				\frac{(xz)^{-\a-1/2}}{t^{2\a+2\b}} \big[ t^2-(x-z)^2\big]^{\a+\b-1/2} \\ & \qquad \times
					\begin{cases}
						\left( \frac{(x+z)^2-t^2}{xz}\right)^{\a+\b-1/2}, & \a + \b < 1/2, \\
						1 + \log\left(\frac{4xz}{(x+z)^2-t^2}\right), & \a + \b = 1/2,\\
						1, & \a + \b > 1/2.		
					\end{cases}
			\end{align*}
	\end{itemize}
			The absolute values in (2a)--(2c) can be suppressed and $\lesssim$ can be replaced by $\simeq$
			if and only if $(\a,\b)$ satisfy neither $\a > 1/2$ and $\b < 0$ nor $\a < -1/2$ and $\b < -2\a$.
\item[(3)]
	The following bounds are uniform in $t > x+z$.
		\begin{itemize}
			\item[(3a)] If $-\b \in \mathbb{N}$, then this part of the kernel vanishes.
			\item[(3b)]
				If $2\a+\b = 0$ and $\b \neq 0$, then
				$$
					K_t^{\a,\b}(x,z) \simeq
					\frac{1}{t^{2\a+2\b}} \big[ t^2-(x-z)^2\big]^{\b-1}
					\bigg( \frac{t^2-(x+z)^2}{t^2-(x-z)^2}\bigg)^{\a+\b-1/2}.
				$$
			\item[(3c)]
				If $\b=1$, then
				$$
					K_t^{\a,\b}(x,z) \simeq
					\frac{1}{t^{2\a+2\b}} \big[ t^2-(x-z)^2\big]^{\b-1}.
				$$
			\item[(3d)]
				For all $(\a,\b)$ not covered by items (3a)--(3c),
					\begin{align*}
					|K_t^{\a,\b}(x,z)| & \lesssim
					\frac{1}{t^{2\a+2\b}} \big[ t^2-(x-z)^2\big]^{\b-1} \\ & \qquad \times
						\begin{cases}
							\left( \frac{t^2-(x+z)^2}{t^2-(x-z)^2}\right)^{\a+\b-1/2}, & \a + \b < 1/2, \\
							1 + \log\left(\frac{t^2-(x-z)^2}{t^2-(x+z)^2}\right), & \a + \b = 1/2, \\
							1, & \a + \b > 1/2.
						\end{cases}						
					\end{align*}
		\end{itemize}
		The relation $\lesssim$ in (3d) can be replaced by $\simeq$
				if and only if $(\a,\b)$ does not satisfy $1< \b < -2\a$. If this is the case, then also
				the absolute value can be suppressed, provided that the kernel is multiplied
				by $(-1)^{\lfloor \b \wedge 0\rfloor}$.
\end{itemize}
\end{thm}

Note that some estimates in Theorem \ref{thm:kerest2} can be written in a simpler way, by plugging in
specific values of the parameters, nevertheless we keep the general formulas for the sake of better
comparison between the cases. Further, in all the cases the estimates are sharp
($\lesssim$ can be replaced by $\simeq$) provided that
$(t,x,z)$ are restricted to certain regions.
More precisely, this happens when $t$ is sufficiently close to $|x-z|$ or $x+z$ or $\infty$, that is
$$
\dist(t,|x-z|) < \varepsilon \sqrt{xz}\quad \textrm{or} \quad \dist(t,x+z) < \varepsilon \sqrt{xz} \quad
	\textrm{or} \quad t > \varepsilon^{-1} \sqrt{xz}
$$
with $\varepsilon > 0$ small enough.

\section{$L^2$-boundedness of the integral operator $M_t^{\a,\b}$} \label{sec:L2}

The estimates of Theorem \ref{thm:kerest2} allow us to verify directly the $L^2$-boundedness of the
integral operator $M_t^{\a,\b}$.
\begin{propo} \label{prop:L2}
Let $\a > -1$ and $\a+\b > -1/2$. Then, for each $t > 0$,
$L^2(d\mu_{\a}) \subset \domain M_t^{\a,\b}$ and $M_t^{\a,\b}$ is bounded on $L^2(d\mu_{\a})$.
\end{propo}

This together with Proposition \ref{prop:coinc} implies the following.
\begin{coro} \label{cor:L2}
Let $\a > -1$ and $\a+\b > -1/2$. Then the operators $\mathcal{M}_t^{\a,\b}$ and $M_t^{\a,\b}$ coincide
on $L^2(d\mu_{\a})$.
\end{coro}

\begin{proof}[{Proof of Proposition \ref{prop:L2}}]
In view of the scaling property of the kernel \eqref{homo}, we may assume $t=1$.
Then the estimates of Theorem \ref{thm:kerest2} imply the following bound uniform in $x$ and $z$
$$
|K_1^{\a,\b}(x,z)| \lesssim
	\begin{cases}
		0, & 1 < |x-z|, \\
		(xz)^{-\a-1/2}[1-(x-z)^2]^{\gamma} \Big( \frac{(x+z)^2-1}{xz}\Big)^{\gamma}, & |x-z| < 1 < x+z, \\
		[1-(x-z)^2]^{-\a-1/2}[1-(x+z)^2]^{\gamma}, & x+z < 1,
	\end{cases}
$$
where $\gamma \in (-1,0)$ is a constant depending on $\a$ and $\b$; actually, taking
$\gamma = (\a+\b-1/2)\wedge (-\varepsilon)$ with a small $\varepsilon > 0$ will suffice ($\varepsilon$ to
take care of the logarithms, otherwise $\varepsilon = 0$ would be enough).
The right-hand side above can be simplified by taking into account the constraints on $x$ and $z$ and
the relations $4xz = (x+z)^2-(x-z)^2 \simeq (x+z) (x+z-|x-z|)$. We get
$$
|K_1^{\a,\b}(x,z)| \lesssim
	\begin{cases}
		0, & |x-z| > 1, \\
		(xz)^{-\a-1/2} \Big( \frac{(1-|x-z|)(x+z-1)}{x+z-|x-z|} \Big)^{\gamma}, & |x-z| < 1, \;\; x+z > 1, \\
		(1-|x-z|)^{-\a-1/2} [1-(x+z)]^{\gamma}, & x+z < 1.
	\end{cases}
$$
We will consider separately the two integral operators defined by the expressions on the right-hand side here.
It is enough to verify that each of them is bounded on $L^2(d\mu_{\a})$.

Let
\begin{align*}
L(x,z) & = \chi_{\{|x-z|<1,\; x+z > 1\}} (xz)^{-\a-1/2}
	\bigg( \frac{(1-|x-z|)(x+z-1)}{x+z-|x-z|} \bigg)^{\gamma}, \\
M(x,z) & = \chi_{\{x+z < 1\}}	(1-|x-z|)^{-\a-1/2} [1-(x+z)]^{\gamma},
\end{align*}
and denote the corresponding integral operators (integration with respect to $d\mu_{\a}$) by
$\mathbf{L}$ and $\mathbf{M}$, respectively. Our strategy to show $L^2(d\mu_{\a})$-boundedness of
$\mathbf{L}$ and $\mathbf{M}$ is mainly based on the Schur test, applied to integral operators
$\widetilde{\mathbf{L}}$ and $\widetilde{\mathbf{M}}$ defined by the kernels
$(xz)^{\a+1/2}L(x,z)$ and $(xz)^{\a+1/2}M(x,z)$, respectively, integration being with respect to
Lebesgue measure $dx$ in $(0,\infty)$. Observe that $\mathbf{L}$ is bounded in $L^2(d\mu_{\a})$ if and only if
$\widetilde{\mathbf{L}}$ is bounded in $L^2(dx)$; similarly for $\mathbf{M}$ and $\widetilde{\mathbf{M}}$.
Recall that, taking into account the positivity and symmetry of our kernels, the Schur test says that
the bound
$$
\int_0^{\infty} (xz)^{\a+1/2}L(x,z)\, dz \lesssim 1, \qquad x > 0,
$$
implies $L^p(dx)$-boundedness, $1 \le p \le \infty$, of $\widetilde{\mathbf{L}}$; analogous implication
holds for $\widetilde{\mathbf{M}}$.

To proceed, we first focus on $\mathbf{L}$, or rather $\widetilde{\mathbf{L}}$.
We have (recall that $-1 < \gamma < 0$)
\begin{align*}
(xz)^{\a+1/2}L(x,z) & = \chi_{\{|x-z|<1,\; x+z > 1\}} \bigg(
	\frac{(1-|x-z|)(x+z-1)}{(1-|x-z|) + (x+z-1)} \bigg)^{\gamma} \\
& \simeq \chi_{\{|x-z|<1,\; x+z > 1\}} \Big[ (1-|x-z|)^{\gamma} + (x+z-1)^{\gamma} \Big] \\
& \lesssim \chi_{\{|x-z|<1\}} (1-|x-z|)^{\gamma} + \chi_{\{|x-z|<1,\; x+z > 1\}} (x+z-1)^{\gamma}.
\end{align*}
In the last sum the first term is given by an integrable convolution kernel, so it defines
an operator bounded on all $L^p(dx)$, $1\le p \le \infty$ (clearly, the Schur test applies as well
with the same conclusion). To deal with the second term, we write
\begin{align*}
\int_0^{\infty} \chi_{\{|x-z|<1,\; x+z > 1\}} (x+z-1)^{\gamma}\, dz & = \int_{|x-1|}^{x+1}
	(x+z-1)^{\gamma}\, dz \\
& \simeq x^{\gamma+1} - [(x-1)\vee 0]^{\gamma+1} \lesssim 1, \qquad x > 0,
\end{align*}
and invoke the Schur test. The $L^p(dx)$-boundedness of $\widetilde{\mathbf{L}}$ follows, and this
implies $L^2(d\mu_{\a})$-boundedness of $\mathbf{L}$.

Next, we analyze $\mathbf{M}$. Here the Schur test gives the conclusion when applied to 
$\widetilde{\mathbf{M}}$, but that with $-1 < \a < -1/2$ excluded.
Therefore, to cover all $\a > -1$, we argue in a more subtle way. In the first
step we will show that the integral operator
$$
\mathbf{M}_1 f(x) = \int_0^{\infty} \chi_{\{|x-z| \le 1/2\}} M(x,z) f(z)\, d\mu_{\a}(z)
$$
is bounded on $L^2(d\mu_{\a})$. Indeed, this follows by using the Schur test, since
\begin{align*}
\int_0^{\infty} \chi_{\{|x-z|\le 1/2\}} M(x,z)\, d\mu_{\a}(z) & \simeq
	\int_0^{\infty} \chi_{\{|x-z| \le 1/2,\, x+z < 1\}} (1-x-z)^{\gamma}z^{2\a+1}\, dz \\
& \le \chi_{\{x < 3/4\}} (1-x)^{\gamma+2\a+2} \int_0^{1} (1-s)^{\gamma} s^{2\a+1}\, ds \\
& \lesssim 1, \qquad x > 0.
\end{align*}
It remains to verify that
$$
\mathbf{M}_2 f(x) = \int_0^{\infty} \chi_{\{x > z\}} \chi_{\{|x-z| > 1/2\}} M(x,z)f(z)\, d\mu_{\a}(z)
$$
is bounded in $L^2(d\mu_{\a})$, because then automatically the same is true for its dual
$\mathbf{M}_2^*$ and $\mathbf{M} = \mathbf{M}_1 + \mathbf{M}_2 + \mathbf{M}_2^*$.

Observe that the kernel of $\mathbf{M}_2$ can be estimated
$$
\chi_{\{x > z\}} \chi_{\{|x-z| > 1/2\}} M(x,z) \lesssim \chi_{\{x > z,\, 1/2 < x < 1,\, z < 1-x\}}
	(1-x)^{-\a-1/2} (1-x-z)^{\gamma},
$$
so it is enough to check the bound
\begin{align*}
& \bigg\| \chi_{\{1/2<x<1\}} (1-x)^{-\a-1/2} \int_0^{\infty} \chi_{\{z < 1-x\}}
	(1-x-z)^{\gamma}f(z)z^{2\a+1}\, dz \bigg\|_{L^2(\mathbb{R}_+,dx)} \\
& \qquad \lesssim \|x^{\a+1/2}f(x)\|_{L^2(\mathbb{R}_+,dx)}, \qquad f \in L^2(d\mu_{\a}).
\end{align*}
By changing the variable $1-x=y$ and letting $F(z) = z^{\a+1/2}f(z)$, we see that this task will be done
once we justify that
$$
\bigg\| \chi_{\{y<1\}}\int_0^{y} \Big(\frac{z}y\Big)^{\a+1/2} (y-z)^{\gamma} F(z)\, dz
 \bigg\|_{L^2(\mathbb{R}_+,dy)}	\lesssim \|F\|_{L^2(\mathbb{R}_+,dx)}, \qquad F \in L^2(\mathbb{R}_+,dx).
$$

Let $\mathcal{I}$ denote the left-hand side in the above estimate. Changing the variable $z=yr$ and then using
Minkowski's integral inequality we get
\begin{align*}
\mathcal{I} & = \bigg\| \chi_{\{y<1\}} y^{\gamma+1} \int_0^1 r^{\alpha+1/2} (1-r)^{\gamma} F(yr)\, dr
\bigg\|_{L^2(\mathbb{R}_+,dy)} \\
& \le \int_0^1 r^{\a+1/2} (1-r)^{\gamma} \big\| \chi_{\{y<1\}} y^{\gamma+1} F(yr) \big\|_{L^2(\mathbb{R}_+,dy)}\, dr.
\end{align*}
To estimate the norm expression under the last integral, we change back the variable $yr=z$ and obtain
$$
\big\| \chi_{\{y<1\}} y^{\gamma+1} F(yr) \big\|_{L^2(\mathbb{R}_+,dy)} = \frac{1}{\sqrt{r}} \bigg(
\int_0^{r} \Big(\frac{z}{r}\Big)^{2(\gamma+1)} |F(z)|^2\, dz \bigg)^{1/2}
\le \frac{1}{\sqrt{r}} \|F\|_{L^2(\mathbb{R}_+,dz)},
$$
since $2(\gamma+1) > 0$. Consequently,
$$
\mathcal{I} \le \|F\|_{L^2(\mathbb{R}_+,dz)} \int_0^1 r^{\a} (1-r)^{\gamma}\, dr \lesssim \|F\|_{L^2(\mathbb{R}_+,dz)}.
$$
This finishes proving $L^2(d\mu_{\a})$-boundedness of $\mathbf{M}_2$, thus also of $\mathbf{M}$.

The proof of Proposition \ref{prop:L2} is now complete.
\end{proof}

\section{Time variable norm estimates of the kernel $K_t^{\a,\b}(x,z)$} \label{sec:test}

In this section we find possibly sharp estimates (which in fact are sharp in some cases, and presumably
sharp in all the cases), of the
norm of $K_t^{\a,\b}(x,z)$ in $L^r(t^{\rho}dt)$. We consider here all power weights $t^{\rho}$,
$\rho \in \mathbb{R}$ and $1 \le r < \infty$. The case $r=\infty$ is not treated. It corresponds
to the maximal operator associated with radial spherical means, whose analysis requires more subtle
methods than those we apply for $r<\infty$; in particular, evaluating first the supremum in $t$ of the
kernel does not lead to satisfactory results.

\begin{thm} \label{thm:norm_est}
Assume that $\a > -1$ and $\a + \b > -1/2$. Let $1 \le r < \infty$ and assume further
\begin{equation} \label{conds}
 \a + \b > \frac{1}2 - \frac{1}{r} \qquad \textrm{and} \qquad
	\Big[ \frac{\rho+1}{r} < 2\a +2 \;\; \textrm{if} \;\; -\b \notin \mathbb{N} \Big].
\end{equation}
Then, uniformly in $x,z > 0$,
\begin{align*}
\big\| K_t^{\a,\b}(x,z) \big\|_{L^r(t^{\rho}dt)} & \lesssim
	(x+z)^{-2\a-1} 
 \left.
		\begin{cases}
			|x-z|^{(\rho+1)/r-1}, & \frac{\rho + 1}{r}< 1\\
			1 + \log^{1/r}\big(\frac{x+z}{|x-z|}\big), & \frac{\rho + 1}{r}= 1\\
			(x+z)^{(\rho+1)/r-1}, & \frac{\rho + 1}{r}> 1
		\end{cases} \right\} \\
& \qquad \times 
\left.
		\begin{cases}
			1, & \b + \frac{1}{r} > 1\\
			1 + \chi_{\{\b \neq 0\}}\log^{1/r}\big( \frac{x}z \vee \frac{z}x\big), & \b + \frac{1}r = 1  \\
			\big( \frac{x}z \wedge \frac{z}x\big)^{\b+ 1/r -1}, & \b + \frac{1}r < 1 
		\end{cases} \right\}.				
\end{align*}

Moreover, the relation $\lesssim$ can be replaced by $\simeq$ in the above estimate if either of the
following statements is true:
\begin{itemize}
\item[(a)]
$\a$ and $\b$ satisfy neither [$\a < -1/2$ and $\b < -2\a$] nor [$\a > 1/2$ and $\b < 0$],
\item[(b)]
$\a$ and $\b$ do not satisfy [$1 < \b < -2\a$ or $-\b \in \mathbb{N}$] and $x,z$ stay non-comparable,
\item[(c)]
$\a$ and $\b$ do not satisfy [$1 < \b < -2\a$ or $-\b \in \mathbb{N}$] and $(\rho+1)/r > 1$.
\end{itemize}

Furthermore, if either of the conditions in \eqref{conds} is not satisfied,
then the $L^r(t^{\rho}dt)$ norm of the kernel is infinite.
\end{thm}

Observe that there are some $\a,\b$
for which the unweighted norm ($\rho=0$) is infinite. This motivates introduction of power weights.

In the proof of Theorem \ref{thm:norm_est} we will use repeatedly the lemma below.
\begin{lema} \label{lem:A}
Let $\gamma > -1$, $\delta \in \mathbb{R}$ and $0 < C < 1$ be fixed.
\begin{itemize}
	\item[(a)] The following relation holds uniformly in $0 < A \le C$,
		$$
			\int_0^{A} w^{\gamma}(1-w)^{\delta}\, dw \simeq A^{\gamma + 1}.
		$$
	\item[(b)] The following bounds hold uniformly in $C \le A < 1$,
		$$
		 	\int_0^{A} w^{\gamma}(1-w)^{\delta}\, dw \simeq
		 		\begin{cases}
		 			1, & \quad \delta > -1, \\
		 			\log\frac{1}{1-A}, & \quad \delta = -1, \\
		 			(1-A)^{\delta+1}, & \quad \delta < -1.
		 		\end{cases}
		$$
\end{itemize}
For $\gamma \le -1$ the integral diverges to infinity for any $0 < A < 1$.
\end{lema}

\begin{proof}
Simple exercise.
\end{proof}

\begin{proof}[{Proof of Theorem \ref{thm:norm_est}}]
We will integrate against $t^{\rho}dt$ the right-hand sides of the bounds in
Theorem \ref{thm:kerest2} raised to power $r$.
We split this integration with respect to the four regions
$$
|x-z| < t < \sqrt{x^2+z^2} < t < x+z < t < \sqrt{2}(x+z) < t,
$$
and denote the resulting (non-negative) integrals by $I_1,I_2,I_3$ and $I_4$, respectively.
These integrals, in general, will be analyzed separately. Then the resulting estimates will be
merged via considering comparable and non-comparable values of $x$ and $z$.

We proceed by considering the most involved situation when items (2c) and (3d)
from Theorem \ref{thm:kerest2} are combined.
It is convenient to distinguish three main cases emerging naturally from the estimates in
Theorem \ref{thm:kerest2} (2c) and (3d).

\noindent \textbf{Case 1.} $\a + \b > 1/2$. We have
\begin{align} \nonumber
I_1 & = (xz)^{-(\a+1/2)r} \int_{|x-z|}^{\sqrt{x^2+z^2}} \big[ t^2 - (x-z)^2\big]^{(\a+\b-1/2)r}
	t^{\rho-2(\a+\b)r}\, dt \\
& = \frac{1}{2} (xz)^{-(\a+1/2)r} |x-z|^{\rho+1-r} \int_0^{\frac{2xz}{x^2+z^2}}
	w^{(\a+\b-1/2)r} (1-w)^{(r-\rho-3)/2}\, dw, \label{pp1}
\end{align}
where the second identity follows by the change of variable $t^2 = (x-z)^2/(1-w)$.
Next, observing that on the interval of integration in $I_2$ one has $t \simeq x+z$ and
$t^2 - (x-z)^2 \simeq xz$, and the length of that interval is
$x+z - \sqrt{x^2+z^2} \simeq xz/(x+z)$, we immediately get
$$
I_2 \simeq (x+z)^{-2(\a+\b)r + \rho-1} (xz)^{(\b-1)r+1}.
$$
Finally, changing the variable $t^2 = (x-z)^2/w$ we arrive at
\begin{align} \nonumber
I_3 + I_4 & = \int_{x+z}^{\infty} \big[t^2-(x-z)^2\big]^{(\b-1)r} t^{\rho-2(\a+\b)r}\, dt \\
& = \frac{1}2 |x-z|^{-2(\a+1)r + \rho +1} \int_0^{\big(\frac{|x-z|}{x+z}\big)^2}
	w^{(\a+1)r + (1-\rho)/2 -2} (1-w)^{(\b-1)r}\, dw. \label{pp2}
\end{align}

Consider now comparable $x$ and $z$. In this situation $xz \simeq (x+z)^2$. Further, since the upper
limit of integration in \eqref{pp1} is separated from $0$, by Lemma \ref{lem:A} (b) we conclude
\begin{equation} \label{I1sim}
I_1 \simeq (x+z)^{-(2\a+1)r} |x-z|^{\rho+1-r}
		\begin{cases}
			1, & \quad r > \rho +1, \\
			1+\log\frac{x+z}{|x-z|}, & \quad r = \rho +1,\\
			\big(\frac{x+z}{|x-z|}\big)^{\rho+1-r}, & \quad r < \rho+1.
		\end{cases}
\end{equation}
Since now
$$
I_2 \simeq (x+z)^{-(2\a+1)r + \rho +1 - r},
$$
it is straightforward to see that $I_2$ is controlled by $I_1$. 
As for $I_3+I_4$, we observe that the upper limit of integration in \eqref{pp2} is separated from $1$ and
apply Lemma \ref{lem:A} (a) to get
$$
I_3+I_4 \simeq (x+z)^{-(2\a+1)r + \rho +1 - r}.
$$
Here the right-hand side is the same as in case of $I_2$.
Summing up, $I_1$ dominates $I_2+I_3+I_4$ and the desired bound for $x \simeq z$ follows. 

Let now $x$ and $z$ be non-comparable. For symmetry reasons, we may assume that $x \gg z$.
Then $x+z \simeq |x-z| \simeq x$. Taking into account that the upper limit of integration in \eqref{pp1}
is separated from $1$ and applying Lemma \ref{lem:A} (a) we obtain
\begin{equation} \label{I1nonsim}
I_1 \simeq x^{-(2\a+1)r} x^{\rho+1-r} \Big(\frac{z}x\Big)^{(\b-1)r+1}.
\end{equation}
Moreover, it is straightforward to see that $I_2$ is comparable to the right-hand side here. 
Passing to $I_3+I_4$, we note that the upper limit of integration in \eqref{pp2} is now separated from
$0$, so Lemma \ref{lem:A} (b) can be applied. This leads to
$$
I_3 + I_4 \simeq x^{-(2\a+1)r} x^{\rho+1-r}
		\begin{cases}
			1, & \quad \frac{1}r > 1-\b, \\
			1+\log\frac{x}z, & \quad \frac{1}r = 1-\b,\\
			\big(\frac{z}x\big)^{(\b-1)r+1}, & \quad \frac{1}r < 1-\b.
		\end{cases}
$$
Since, as easily verified, $I_3+I_4$ controls $I_1+I_2$, we get the bound asserted in the theorem for
$x \gg z$. The conclusion follows.

\noindent \textbf{Case 2.} $\a + \b < 1/2$. Since on the interval of integration in $I_1$ we have
$(x+z)^2-t^2 \simeq xz$, we get
$$
I_1 \simeq (xz)^{-(\a+1/2)r} \int_{|x-z|}^{\sqrt{x^2+z^2}} \big[ t^2 - (x-z)^2\big]^{(\a+\b-1/2)r}
	t^{\rho-2(\a+\b)r}\, dt.
$$
The right-hand side here corresponds to $I_1$ from Case 1, so all the bounds for the present $I_1$ will
be as in Case 1. Considering $I_2$, we observe that on the interval of integration
$t^2-(x-z)^2 \simeq xz$ and $t \simeq x+z$. Consequently,
\begin{align} \nonumber
I_2 & \simeq (xz)^{-(\a+1/2)r}(x+z)^{\rho-2(\a+\b)r-1} \int_{\sqrt{x^2+z^2}}^{x+z} 
	\big[(x+z)^2-t^2\big]^{(\a+\b-1/2)r} t\, dt \\
& = \frac{1}{2} (xz)^{-(\a+1/2)r} (x+z)^{\rho+1-r} \int_0^{\frac{2xz}{(x+z)^2}}	w^{(\a+\b-1/2)r}\, dw,
\label{pp3}
\end{align}
where the last identity follows by the change of variable $t^2 = (x+z)^2(1-w)$.
For $I_3$ we notice that on the interval of integration $t \simeq x+z$ and write
\begin{align} \nonumber
I_3 & \simeq (x+z)^{-2(\a+\b)r+\rho-1} \int_{x+z}^{\sqrt{2}(x+z)} \big[t^2-(x-z)^2\big]^{-(\a+1/2)r}
	\big[ t^2-(x+z)^2\big]^{(\a+\b-1/2)r} t\, dt \\
& \simeq (x+z)^{-(2\a+1)r+\rho+1-r} \bigg( \frac{xz}{(x+z)^2}\bigg)^{(\b-1)r+1}
	\int_0^{\frac{(x+z)^2}{(x+z)^2+4xz}} w^{(\a+\b-1/2)r}(1-w)^{(1-\b)r-2}\, dw,	\label{pp4}
\end{align}
where the last relation is obtained by changing the variable 
\begin{equation} \label{chv3}
t^2 = (x+z)^2+4xz \frac{w}{1-w}.
\end{equation}
Finally, on the interval of integration of $I_4$ one has $t^2-(x-z)^2 \simeq t^2-(x+z)^2 \simeq t^2$ and
therefore
$$
I_4 \simeq \int_{\sqrt{2}(x+z)}^{\infty} t^{-(2\a+1)r + \rho -r} \, dt \simeq (x+z)^{-(2\a+1)r+\rho+1-r}.
$$

Let now $x$ and $z$ be comparable. Then $I_1$ satisfies \eqref{I1sim}. Further, the upper limit of
integration in \eqref{pp3} is separated from $0$, hence that integral has the size of a (positive and finite)
constant. Therefore, using also $xz \simeq (x+z)^2$,
$$
I_2 \simeq (x+z)^{-(2\a+1)r + \rho+1-r}
$$
which, as in Case 1, is controlled by $I_1$. Clearly, the latter is also true for $I_4$.
As for $I_3$, the upper limit of integration is separated both from $0$ and $1$, so we get
the same behavior as for $I_2$ and $I_4$. Altogether, this gives the relevant bound for $x \simeq z$.

When $x$ and $z$ are non-comparable, say $x \gg z$, we argue in a similar way. $I_1$ satisfies
\eqref{I1nonsim}. To estimate $I_2$ we just integrate in \eqref{pp3} getting the same behavior as for $I_1$,
so
$$
I_1+I_2 \simeq x^{-(2\a+1)r} x^{\rho+1-r} \Big(\frac{z}{x}\Big)^{(\b-1)r+1}.
$$
The upper limit of integration in \eqref{pp4} is separated from $0$, so applying Lemma \ref{lem:A} (b) we
arrive at
$$
I_3 \simeq x^{-(2\a+1)r} x^{\rho+1-r}
		\begin{cases}
			1, & \quad \frac{1}r > 1-\b, \\
			1+\log\frac{x}z, & \quad \frac{1}r = 1-\b,\\
			\big(\frac{z}x\big)^{(\b-1)r+1}, & \quad \frac{1}r < 1-\b.
		\end{cases}
$$
Since $I_4 \simeq x^{-(2\a+1)r} x^{\rho+1-r}$, we see that $I_4 \lesssim I_3$. Moreover, as easily
verified, $I_1+I_2 \lesssim I_3$. Thus the estimate of the theorem follows for $x \gg z$, and by
symmetry also for $x \ll z$.

\noindent \textbf{Case 3.} $\a+\b = 1/2$. To treat $I_1$, we notice that on the interval of integration
$(x+z)^2-t^2 \simeq xz$, hence the logarithm can be neglected and the estimates for $I_1$ are as in
Cases 1 and 2.
Dealing with $I_2$, we take into account that $t \simeq x+z$ on the interval of integration and get
\begin{align*}
I_2 & \simeq (xz)^{-(\a+1/2)r} (x+z)^{\rho-r-1} \int_{\sqrt{x^2+z^2}}^{x+z} 
	\log^r\bigg(\frac{8xz}{(x+z)^2-t^2}\bigg) t\, dt \\
& \simeq (xz)^{-(\a+1/2)r+1} (x+z)^{\rho-r-1} \int_0^{1/4} \log^r \frac{1}w \, dw.
\end{align*}
Here the second relation is obtained by changing the variable $[(x+z)^2-t^2]/(8xz) = w$, and the
last integral is a (positive and finite) constant depending only on $r$.
In $I_3$ we still have $t \simeq x+z$, thus
\begin{align*}
I_3 & \simeq (x+z)^{\rho-r-1} \int_{x+z}^{\sqrt{2}(x+z)} \big[t^2-(x-z)^2\big]^{(\b-1)r}
	\log^r\bigg( 2\frac{t^2-(x-z)^2}{t^2-(x+z)^2}\bigg) t\, dt \\
& = 2 (xz)^{-(\a+1/2)r+1}(x+z)^{\rho-r-1} \int_0^{\frac{(x+z)^2}{(x+z)^2+4xz}}
	(1-w)^{(1-\b)r-2} \log^r\frac{2}{w} \, dw,
\end{align*}
where the last identity follows by the change of variable \eqref{chv3} and the equality $\b-1=-(\a+1/2)$.
Observe that in the last integral
the logarithmic factor can be neglected for our purpose because it is integrable near $0$, and the upper limit
of integration is always separated from $0$; thus Lemma \ref{lem:A} (b) is applicable. Finally, in $I_4$,
$t^2-(x-z)^2 \simeq t^2-(x+z)^2 \simeq t^2$, so its behavior is the same as in Case 2.

From here we proceed similarly as in Cases 1 and 2 to see that $I_1$ is the dominating integral
when $x \simeq z$, whereas for non-comparable $x$ and $z$ the dominating one is $I_3$.
Combining then the behaviors of $I_1$ and $I_3$ we conclude the desired estimate in Case 3.

Proving the bound of Theorem \ref{thm:norm_est} is finished in the most involved situation when
estimates of items (2c) and (3d) in Theorem \ref{thm:kerest2} are combined.
Other combinations of items (2a)--(2c) and (3a)--(3d)
are implicitly contained in the analysis done so far, since the other bounds coincide with subcases occurring
in (2c) and (3d). Further details are straightforward and thus omitted.

Tracing this proof reveals that the conditions \eqref{conds} are indeed necessary to assure integrability
in various places. If either of them would not be satisfied, then we would have
$I_1+I_2+I_3+I_4 = \infty$. Moreover, since the bounds of Theorem \ref{thm:kerest2} are sharp
when $t \to |x-z|$ or $t \to x+z$ or $t \to \infty$, we infer that the $L^r(t^{\rho}dt)$ norm
of the kernel is infinite if \eqref{conds} does not hold.

Finally, conditions (a)--(c) allowing to replace $\lesssim$ by $\simeq$ are deduced from the corresponding
comments in Theorem \ref{thm:kerest2} and mutual relations between the integrals $I_1,\ldots,I_4$.
More precisely, for showing (b) and (c) the following fact is relevant:
assuming $-\b \notin \mathbb{N}$,  the sum $I_1+\ldots +I_4$ is controlled
by $I_3+I_4$ when $x$ and $z$ stay non-comparable or $r < \rho+1$.
\end{proof}

\section{Mixed norm estimates for $M_t^{\a,\b}$} \label{sec:mixed}

In this section our aim is to study boundedness of $M_t^{\a,\b}$ from $L^p(d\mu_{\a})$
to the mixed norm space $L^q(L^r_{t^{\rho}})(d\mu_{\a})$. More generally, we are interested
in two-weight mixed norm estimates of the form
\begin{equation} \label{mixnorm}
\Big\| \big\| M_t^{\a,\b}f(x)\big\|_{L^r(t^{\rho}dt)} x^{-B} \Big\|_{L^q(d\mu_{\a})} \lesssim
\big\| f(x) x^{A} \big\|_{L^p(d\mu_{\a})},
\end{equation}
which are uniform in $f$. Our objective is to find possibly wide ranges of the parameters
$\a,\b,A,B,r,\rho,p,q$ for which \eqref{mixnorm} holds. Here, in general, we consider
\begin{equation} \label{ranges}
\a > -1, \quad \b > -\a-1/2, \quad 1 \le p,q \le \infty, \quad 1 \le r < \infty,
\quad A,B,\rho \in \mathbb{R}.
\end{equation}

The result below is a simple consequence of homogeneity of the kernel $K_t^{\a,\b}(x,z)$, see
\eqref{homo}.
\begin{propo} \label{prop:nec}
Assume that the parameters satisfy \eqref{ranges}. Then the condition
\begin{equation} \label{nec}
\frac{1}{q} = \frac{1}p +\frac{1}{2\a+2}\Big( A+B-\frac{\rho+1}r\Big)
\end{equation}
is necessary for \eqref{mixnorm} to hold uniformly in, say, $f \in C_c^{\infty}(0,\infty)$.
\end{propo}
Notice that condition \eqref{nec} is independent of $\b$.

To proceed, we shall consider a positive kernel $K^{\a,\b}_{r,\rho}(x,z)$ defined by the right-hand side of
the bound from Theorem \ref{thm:norm_est}, without assuming \eqref{conds}, and will analyze the corresponding
positive operator $K^{\a,\b}_{r,\rho}$.

\subsection{Analysis of the auxiliary operator $K^{\a,\b}_{r,\rho}$}
Let
\begin{align*}
& K_{r,\rho}^{\a,\b}(x,z) = \\ &(x+z)^{-2\a-1} 
 \left.
		\begin{cases}
			|x-z|^{(\rho+1)/r-1}, & \frac{\rho + 1}{r}< 1\\
			1 + \log^{1/r}\big(\frac{x+z}{|x-z|}\big), & \frac{\rho + 1}{r}= 1\\
			(x+z)^{(\rho+1)/r-1}, & \frac{\rho + 1}{r}> 1
		\end{cases} \right\}
\left.
		\begin{cases}
			1, & \b + \frac{1}{r} > 1\\
			1 + \chi_{\{\b \neq 0\}}\log^{1/r}\big( \frac{x}z \vee \frac{z}x\big), & \b + \frac{1}r = 1  \\
			\big( \frac{x}z \wedge \frac{z}x\big)^{\b+ 1/r -1}, & \b + \frac{1}r < 1 
		\end{cases} \right\},				
\end{align*}
and consider the associated integral operator
$$
K_{r,\rho}^{\a,\b}f(x) = \int_0^{\infty} K_{r,\rho}^{\a,\b}(x,z)f(z)\, d\mu_{\a}(z).
$$
We denote by $\domain K_{r,\rho}^{\a,\b}$ the natural domain of this operator, that is the set of all those
functions $f$ for which the defining integral converges for a.a.\ $x$.
Observe that a necessary condition for $\domain K_{r,\rho}^{\a,\b}$ to be non-trivial is $\rho > -1$,
since otherwise for each $x > 0$ the factor $|x-z|^{(\rho+1)/r-1}$ is not locally integrable around $x$.

We will show the following sharp result.
\begin{thm} \label{thm:K}
Assume that $\a,\b,A,B,r,\rho,p,q$ are as in \eqref{ranges}.
\begin{itemize}
\item[(i)]
The inclusion $L^p(x^{Ap}d\mu_{\a}) \subset \domain K^{\a,\b}_{r,\rho}$ holds if and only if $\rho > -1$ and
\begin{align} \label{Cincl}
& \frac{\rho+1}r - \frac{2\a+2}p - \Big(\b + \frac{1}r-1\Big)\wedge 0 < A 
	< \frac{2\a+2}{p'} + \Big(\b + \frac{1}r-1\Big)\wedge 0 \\
& \qquad \Big(\textrm{both}\; \le \; \textrm{when}\; p=1 \; \textrm{and}\; 
	\Big[ \b=0 \; \textrm{or}\; \b+\frac{1}r-1 \neq 0\Big] \Big).
	\nonumber
\end{align}
\item[(ii)]
The estimate
$$
\big\| x^{-B} K^{\a,\b}_{r,\rho}f \big\|_{L^q(d\mu_{\a})} \lesssim \big\| x^A f\big\|_{L^p(d\mu_{\a})}
$$
holds uniformly in $f \in L^p(x^{Ap}d\mu_{\a})$ if and only if the following conditions are satisfied:
\begin{itemize}
\item[(C1)] $p \le q$,
\item[(C2)] $\frac{1}q = \frac{1}p +\frac{1}{2\a+2}\big(A+B-\frac{\rho+1}r\big)$,
\item[(C3)] $\big( A - \frac{2\a+2}{p'} \big) \vee \big(B - \frac{2\a+2}{q} \big) < 
	\big(\b + \frac{1}r-1\big) \wedge 0$ \\
		\quad ($\le$ when $p=q'=1$ and [$\b=0$ or $\b+\frac{1}r\neq 1$]),
\item[(C4)] $\frac{1}q \ge \frac{1}p - \frac{\rho+1}r$ \quad ($>$ when $p=1$ or $q=\infty$).
\end{itemize}
\end{itemize}
\end{thm}
Note that in general none of conditions (C1)--(C4) follows from the others. However, (C4) is
superfluous when $\frac{\rho+1}r > 1$, and in case $\frac{\rho+1}r=1$ it is equivalent to
$(p,q) \neq (1,\infty)$. Moreover, in view of (C2), condition (C4) can be replaced by
\begin{itemize}
\item[(C4')]
$A+B \ge (2\a+1)\Big(\frac{1}q-\frac{1}p\Big)$ \quad ($>$ when $p=1$ or $q=\infty$).
\end{itemize}
Finally, notice that (C2) is exactly \eqref{nec} from Proposition \ref{prop:nec}.

In order to prove Theorem \ref{thm:K} we now define auxiliary positive operators into which
$K_{r,\rho}^{\a,\b}$ will be `decomposed'.
Observe that
\begin{align*}
K_{r,\rho}^{\a,\b}(x,z) & \simeq (x+z)^{-2\a-1} (x+z)^{(\rho+1)/r-1}
	\left.
		\begin{cases}
			1, & \b + \frac{1}{r} > 1\\
			1 + \chi_{\{\b \neq 0\}}\log^{1/r}\big( \frac{x}z \vee \frac{z}x\big), & \b + \frac{1}r = 1  \\
			\big( \frac{x}z \wedge \frac{z}x\big)^{\b+ 1/r -1}, & \b + \frac{1}r < 1 
		\end{cases} \right\} \\
& \qquad + \chi_{\{x/2 < z < 2x\}} z^{-2\a-1}\bigg[ \chi_{\{\frac{\rho+1}r< 1\}}	|x-z|^{(\rho+1)/r-1}
	+ \chi_{\{\frac{\rho+1}r = 1\}}	\log^{1/r}\frac{x+z}{|x-z|} \bigg],
\end{align*}
uniformly in $x,z>0$. Accordingly, for $\eta \in \mathbb{R}$, define
\begin{align*}
H_0^{\eta}f(x) & = x^{-2\a-2 + (\rho+1)/r - \eta} \int_0^x z^{2\a+1+\eta}f(z)\, dz, \\
H_{\infty}^{\eta}f(x) & = x^{\eta} \int_x^{\infty} z^{(\rho+1)/r-1-\eta}f(z)\, dz, \\
H_0^{\textrm{log}}f(x) & = x^{-2\a-2 + (\rho+1)/r} 
	\int_0^{x} \log^{1/r}\Big(\frac{2x}z\Big) z^{2\a+1}f(z)\, dz,\\
H_{\infty}^{\textrm{log}}f(x) & = \int_x^{\infty} \log^{1/r}\Big(\frac{2z}x\Big) z^{(\rho+1)/r-1}f(z)\, dz,\\
Tf(x) & = \int_{x/2}^{2x} |x-z|^{(\rho+1)/r-1}f(z)\, dz,\\
Sf(x) & = \int_{x/2}^{2x} \log^{1/r}\Big( \frac{x+z}{|x-z|}\Big) f(z)\, dz.
\end{align*}
Then the following relation is uniform both pointwise and in $f \ge 0$:
\begin{align*}
K_{r,\rho}^{\a,\b}f & \simeq H_0^{\b+1/r-1}f + H_{\infty}^{\b+1/r-1}f + 
	\chi_{\{\b+1/r>1\}}(H_0^0f + H_{\infty}^0f) \\
& \qquad + \chi_{\{\b+1/r = 1, \b \neq 0\}}
		(H_0^{\textrm{log}}f + H_{\infty}^{\textrm{log}}f) + \chi_{\{\frac{\rho+1}r < 1\}} Tf
			+ \chi_{\{\frac{\rho+1}r =1\}} Sf.
\end{align*}

It is clear that $K_{r,\rho}^{\a,\b}$ is bounded from $L^p(x^{Ap}d\mu_{\a})$ to $L^q(x^{-Bq}d\mu_{\a})$
(or well defined on $L^p(d\mu_{\a})$) if and only if all of the component operators appearing
on the right-hand side above have the property. Therefore we now analyze each of these operators.
We will argue similarly as in \cite[Section 4.1]{NoSt}. Our main tool will be the following characterization
of two power-weight $L^p-L^q$ inequalities for the Hardy operator and its dual; see e.g.\ \cite{B,SS} and
\cite[Lemma 4.1]{NoSt}.
\begin{lema} \label{lem:Har}
Let $a,b \in \mathbb{R}$ and let $1 \le p,q \le \infty$.
\begin{itemize}
\item[(a)]
The estimate
$$
\bigg\| x^b \int_0^x g(y)\, dy \bigg\|_{L^q(\mathbb{R}_+,dx)} 
	\lesssim \big\|x^a g\big\|_{L^p(\mathbb{R}_+,dx)}
$$
holds uniformly in 
$g \in L^p(\mathbb{R}_+,x^{ap}dx)$ if and only if $p \le q$ and
$a - \frac{1}{p'} = b+\frac{1}q$ and $a< \frac{1}{p'}$ ($\le$ in case $p=q'=1$). 
\item[(b)]
The estimate
$$
\bigg\| x^b \int_x^{\infty} g(y)\, dy \bigg\|_{L^q(\mathbb{R}_+,dx)} 
	\lesssim \big\|x^a g\big\|_{L^p(\mathbb{R}_+,dx)}
$$
holds uniformly in $g \in L^p(\mathbb{R}_+, x^{ap}dx)$ if and only if $p \le q$ and
$a - \frac{1}{p'} = b+\frac{1}q$ and $b>-\frac{1}q$ ($\ge$ in case $p=q'=1$). 
\end{itemize}
\end{lema}

For further reference, we state the following conditions:
\begin{itemize}
\item[(C5)] $A < \frac{2\a+2}{p'} +\Big(\b + \frac{1}r-1\Big)\wedge 0 \quad 
	(\le \;\textrm{when}\; p=q'=1 \; \textrm{and}\; [\b=0\; \textrm{or}\; \b+\frac{1}r\neq 1])$,
\item[(C5a)] $A < \frac{2\a+2}{p'} +\b + \frac{1}r-1 \quad (\le \;\textrm{when}\; p=q'=1)$,
\item[(C5b)] $A < \frac{2\a+2}{p'} \quad (\le \;\textrm{when}\; p=q'=1)$,
\item[(C5c)] $A < \frac{2\a+2}{p'}$,
\item[(C6)] $B < \frac{2\a+2}{q} + \Big(\b + \frac{1}r-1\Big) \wedge 0 \quad 
	(\le \;\textrm{when}\; p=q'=1 \; \textrm{and}\; [\b=0\; \textrm{or}\; \b+\frac{1}r\neq 1])$,
\item[(C6a)] $B < \frac{2\a+2}{q} + \b + \frac{1}r-1 \quad (\le \;\textrm{when}\; p=q'=1)$,
\item[(C6b)] $B < \frac{2\a+2}{q} \quad (\le \;\textrm{when}\; p=q'=1)$,
\item[(C6c)] $B < \frac{2\a+2}{q}$.
\end{itemize}
Notice that (C5) and (C6) together are equivalent to (C3) from Theorem \ref{thm:K}.

\noindent \textbf{Analysis of} $\boldsymbol{H_0^{\b+1/r-1}}$.
Substituting $f(z) = z^{-2\a-\b-1/r}g(z)$ we see that the estimate
$$
\big\|x^{-B} H_0^{\b+1/r-1}f\big\|_{L^q(d\mu_{\a})} \lesssim \|x^{A}f\|_{L^p(d\mu_{\a})}
$$
is equivalent to
$$
\bigg\| x^{-2\a-\b+\rho/r-1-B+(2\a+1)/q}\int_0^x g(z)\, dz \bigg\|_{L^q(\mathbb{R}_+,dx)} \lesssim
\big\| x^{-2\a-\b-1/r+A+(2\a+1)/p}g\big\|_{L^p(\mathbb{R}_+,dx)}.
$$
By Lemma \ref{lem:Har} (a), this holds if and only if (C1), (C2) and (C5a) hold simultaneously.

\noindent \textbf{Analysis of} $\boldsymbol{H_{\infty}^{\b+1/r-1}}$.
Substituting $f(z) = z^{\b-\rho/r}g(z)$ we can write the estimate
$$
\big\|x^{-B} H_{\infty}^{\b+1/r-1}f\big\|_{L^q(d\mu_{\a})} \lesssim \|x^{A}f\|_{L^p(d\mu_{\a})}
$$
in the equivalent form
$$
\bigg\| x^{\b+1/r-1-B+(2\a+1)/q}\int_x^{\infty} g(z)\, dz \bigg\|_{L^q(\mathbb{R}_+,dx)} \lesssim
\big\| x^{\b-\rho/r+A+(2\a+1)/p}g\big\|_{L^p(\mathbb{R}_+,dx)}.
$$
Applying Lemma \ref{lem:Har} (b), we see that this holds if and only if (C1), (C2) and (C6a)
hold simultaneously.

Clearly, this analysis of $H_0^{\b+1/r-1}$ and $H_{\infty}^{\b+1/r-1}$ is valid for any real
$\eta = \b + 1/r -1$.

\noindent \textbf{Analysis of} $\boldsymbol{H_0^{0}}$ \textbf{and} $\boldsymbol{H_{\infty}^{0}}$.
This is a special case of the above, so we infer that
$$
\big\|x^{-B} H_0^{0}f\big\|_{L^q(d\mu_{\a})} \lesssim \|x^{A}f\|_{L^p(d\mu_{\a})}
$$
holds if and only if (C1), (C2) and (C5b) are satisfied. Further,
$$
\big\|x^{-B} H_{\infty}^{0}f\big\|_{L^q(d\mu_{\a})} \lesssim \|x^{A}f\|_{L^p(d\mu_{\a})}
$$
holds if and only if (C1), (C2) and (C6b) are satisfied.

\noindent \textbf{Analysis of} $\boldsymbol{H_0^{\log}}$ \textbf{and} $\boldsymbol{H_{\infty}^{\log}}$.
We observe that, given any $\eta < 0$,
\begin{equation} \label{log_c}
H_0^{0}f \lesssim H_0^{\log}f \lesssim H_0^{\eta}f \quad \textrm{and} \quad
H^0_{\infty}f \lesssim H_{\infty}^{\log}f \lesssim H_{\infty}^{\eta}f, \qquad f \ge 0.
\end{equation}
This implies that conditions for boundedness of the logarithmic operators are (C1), (C2) and either (C5a)
or (C6a), with $\eta = \b+1/r-1 =0$, but excluding the case $p=q'=1$ in which weak inequality appears
in (C5a) and (C6a).
Thus, assuming for a moment that $(p,q) \neq (1,\infty)$, we see that
\begin{equation} \label{log1}
\big\|x^{-B} H_0^{\log}f\big\|_{L^q(d\mu_{\a})} \lesssim \|x^{A}f\|_{L^p(d\mu_{\a})}
\end{equation}
holds if and only if (C1), (C2) and (C5c) are satisfied, and similarly
\begin{equation} \label{log2}
\big\|x^{-B} H_{\infty}^{\log}f\big\|_{L^q(d\mu_{\a})} \lesssim \|x^{A}f\|_{L^p(d\mu_{\a})}
\end{equation}
holds if and only if (C1), (C2) and (C6c) are satisfied.

The remaining case requires further treatment.
Assuming (C1), (C2) and that $(p,q) = (1,\infty)$, it is easily seen directly that \eqref{log1} does not
hold when $A = \frac{2\a+2}{p'}=0$ and, similarly, \eqref{log2} is not true if $B = \frac{2\a+2}q=0$.
So, in general, here the boundedness conditions are (C1), (C2), and either (C5c) or (C6c), respectively.

\noindent \textbf{Analysis of} $\boldsymbol{T}$ \textbf{in case} $\boldsymbol{\frac{\rho+1}r < 1}$
\textbf{and} $\boldsymbol{S}$ \textbf{in case} $\boldsymbol{\frac{\rho+1}r = 1}$.
Here we may assume that (C1) and (C2) are satisfied, and $\rho > -1$.
We can then invoke the analysis of $T$ and $S$ performed
in \cite[Section 4.1]{NoSt}, with the quantity $\frac{\rho+1}{2r}$
playing the role of $\sigma$ from \cite{NoSt}.
To be precise, here the bound $\frac{\rho+1}{2r} < \a +1$ may not be satisfied, but
that does not affect the arguments in question. Thus the conclusion is that
(C4) from Theorem \ref{thm:K} is necessary and sufficient
(under the assumptions made) both for $T$ and $S$, separately,
to be bounded from $L^p(x^{Ap}d\mu_{\a})$ to $L^q(x^{-Bq}d\mu_{\a})$.

We can finally prove the theorem.
\begin{proof}[{Proof of Theorem \ref{thm:K}}]
To prove (i), we proceed as in the proof \cite[Theorem 2.5(i)]{NoSt}. From there, we know that
$T$ is well defined on $L^p(x^{Ap}d\mu_{\a})$ whenever $\sigma = \frac{\rho+1}{2r}>0$,
and $S$ is always well defined on $L^p(x^{Ap}d\mu_{\a})$. So one has to look at
the Hardy type operators $H_0^{\eta}, H_{\infty}^{\eta}, H_0^{\log}, H_{\infty}^{\log}$.

Arguing as in \cite{NoSt} we find that the condition
$$
\frac{\rho+1}r - \frac{2\a+2}{p} - \eta < A < \frac{2\a+2}{p'} + \eta
\quad (\textrm{both}\; \le \; \textrm{if}\; p=1)
$$
is necessary and sufficient for the sum $H_0^{\eta}+ H_{\infty}^{\eta}$ to be well defined on
$L^p(x^{Ap}d\mu_{\a})$. Further, using \eqref{log_c}, we infer that the same condition with $\eta = 0$
is necessary
and sufficient for the sum $H_0^{\log}+H_{\infty}^{\log}$ to be well defined on $L^p(x^{Ap}d\mu_{\a})$,
but now without weakening the inequalities in case $p=1$. The latter is easily verified directly,
by means of suitable counterexamples. Summing up,
$$
H_0^{\b+1/r-1}+H_{\infty}^{\b+1/r-1} + \chi_{\{\b+1/r>1\}} (H_0^{0}+H_{\infty}^0)
	+ \chi_{\{\b+1/r=1,\b\neq 0\}}(H_0^{\log}+H_{\infty}^{\log})
$$
is well defined on $L^p(x^{Ap}d\mu_{\a})$ if and only if \eqref{Cincl} holds. The conclusion follows.

The proof of (ii) is essentially contained in the analysis of
$H_0^{\b+1/r-1}$, $H_{\infty}^{\b+1/r-1}$, $H_0^0$, $H_{\infty}^0$,
$H_0^{\log}$, $H_{\infty}^{\log}$, $T$, $S$
done above. Observe that (C5a), (C5b) when $\b+1/r > 1$ and
(C5c) when $\b+1/r=1$ and $\b \neq 0$ altogether are equivalent to (C5).
Analogous observation pertains to (C6a), (C6b), (C6c) and (C6).
\end{proof}

\begin{remark}
Perhaps a bit surprisingly,
the operator $K^{\a,\b}_{r,\rho}$ resembles much the potential operator related
to the modified Hankel transform,
see \cite[Section 2.1]{NoSt}. Even more, if $\sigma = \frac{\rho+1}{2r} < \a + 1$ and
[$\b + 1/r -1 > 0$ or $(\b,r)=(0,1)$],
then the kernel
$K_{r,\rho}^{\a,\b}(x,z)$ behaves exactly like the potential kernel $K^{\a,\sigma}(x,z)$ and one gets the
same boundedness results for the associated operators.
\end{remark}

We now comment on a shape of the set of all $1\le p,q \le \infty$ for which the estimate of
Theorem \ref{thm:K} holds. Actually, we are going to look at the corresponding set $D$ of pairs
$(\frac{1}p,\frac{1}q)$ in the closed unit square $[0,1]^2$. Then, depending on the parameters involved,
the following situations occur (and no others).
\begin{itemize}
\item[(S1)] $D$ is a segment in the lower triangle $\frac{1}p > \frac{1}q$, parallel to the diagonal
and with endpoints included and located on the boundary of the square.
\item[(S2)] $D$ is a subsegment of that from (S1)
	having excluded any endpoint not lying on the boundary of the square.
\item[(S3)] $D$ is a subsegment of the diagonal of $[0,1]^2$, of length strictly smaller than the diagonal,
	having excluded any endpoint not lying on the boundary of the square.
\item[(S4)] $D$ is just one point, the lower-right vertex of the square.
\item[(S5)] $D$ is empty.
\end{itemize}
All segments of types (S1)--(S3) indeed occur with suitable choices of the parameters.

\subsection{Main results}
By homogeneity of the kernel \eqref{homo} it follows that a power-weighted space $L^p(x^{Ap}d\mu_{\a})$
is included in $\domain M_t^{\a,\b}$ for a given $t>0$ if and only if
$L^p(x^{Ap}d\mu_{\a}) \subset \domain M_t^{\a,\b}$ for all $t>0$. Thus from Theorems \ref{thm:norm_est}
and \ref{thm:K} (i) we conclude that all the weighted $L^p(d\mu_{\a})$ spaces admitted in Theorem
\ref{thm:K} (i) are contained, under all the relevant assumptions on the parameters, in
$\domain M_t^{\a,\b}$, all $t>0$.

The main result of this paper is a straightforward consequence of Theorems \ref{thm:norm_est}
and \ref{thm:K}. It reads as follows.
\begin{thm} \label{thm:main}
Let $\a > -1$, $\b > -\a-1/2$, $1\le p,q \le \infty$, $1\le r < \infty$, $A,B,\rho \in \mathbb{R}$ and
assume that
$$
 \a + \b > \frac{1}2 - \frac{1}{r} \qquad \textrm{and} \qquad
	\Big[ \frac{\rho+1}{r} < 2\a +2 \;\; \textrm{if} \;\; -\b \notin \mathbb{N} \Big].
$$
Then, under the conditions (C1)--(C4), $L^p(x^{Ap}d\mu_{\a}) \subset \domain M_t^{\a,\b}$ for each $t>0$ and the estimate
$$
\Big\| \big\| M_t^{\a,\b}f(x)\big\|_{L^r(t^{\rho}dt)} x^{-B} \Big\|_{L^q(d\mu_{\a})} \lesssim
\big\| f(x) x^{A} \big\|_{L^p(d\mu_{\a})}
$$
holds uniformly in $f \in L^p(x^{Ap}d\mu_{\a})$.
\end{thm}

\begin{remark} \label{rem:min}
The order of taking the norms in the mixed norm expression in Theorem \ref{thm:main} can be exchanged.
Indeed, in view of Minkowski's integral inequality, when $q \le r$,
$$
	\Big\| \big\| M_t^{\a,\b}f(x) x^{-B} \big\|_{L^q(d\mu_{\a})} \Big\|_{L^r(t^{\rho}dt)} \le
  \Big\| \big\| M_t^{\a,\b}f(x)\big\|_{L^r(t^{\rho}dt)} x^{-B} \Big\|_{L^q(d\mu_{\a})}.
$$
\end{remark}

From Theorem \ref{thm:main} we conclude immediately the following two-weight mixed norm estimate for the generalized
spherical mean Radon transform $M^{\b}$ in $\mathbb{R}^n$, $n \ge 1$.
\begin{coro} \label{cor:Rn}
Let $n \ge 1$. Then, under the assumptions and conditions of Theorem \ref{thm:main} on the parameters,
with $\a = n/2-1$, one has the estimate
$$
\Big\| \big\| M^{\b}f(x,t)\big\|_{L^r(t^{\rho}dt)} |x|^{-B} \Big\|_{L^q(\mathbb{R}^n,dx)}
\lesssim \big\| f(x) |x|^{A} \big\|_{L^p(\mathbb{R}^n,dx)}, 
	\qquad f \in L^p_{\textrm{rad}}(\mathbb{R}^n,|x|^{Ap}dx).
$$
Here $M^{\b}f(\cdot,t)$ is understood as the extension given by $M_t^{\a,\b}$ of this operator defined
initially on $L^2(\mathbb{R}^n,dx) \cap L^p_{\textrm{rad}}(\mathbb{R}^n,|x|^{Ap}dx)$ by means of the Fourier
transform.
\end{coro}

Theorem \ref{thm:main} and Corollary \ref{cor:Rn} are new results, and that even when specified to an unweighted
setting ($A=B=0$) or to non-generalized radial spherical means ($\b=0$). Moreover, given our strategy of proof, these
results are pretty precise (if not sharp, at least in some cases), in view of the sharpness statements in
Theorem \ref{thm:kerest2} (see also the relevant comment succeeding this theorem) and Theorem \ref{thm:norm_est},
and the optimal result contained in Theorem \ref{thm:K}.

\section{Applications to some PDE problems: weighted Strichartz estimates} \label{sec:appl}

Let $\a > -1$, $\b > -\a-1/2$ and assume that $u(x,t) = c M_t^{\a,\b}f(x)$,
$(x,t) \in \mathbb{R}_+\times \mathbb{R}_+$, is a (weak) solution to a PDE problem with $f$ entering an
initial value condition. Then Theorem \ref{thm:main} and Remark \ref{rem:min} imply a weighted
Strichartz type estimate
$$
\Big\| \big\| u(x,t) x^{-B} \big\|_{L^q(d\mu_{\a})} \Big\|_{L^r(t^{\rho}dt)} \lesssim
	\big\| f(x) x^A \big\|_{L^p(d\mu_{\a})},
$$
under the assumptions and conditions of Theorem \ref{thm:main} and provided that $q \le r$.
Moreover, if $u(x,t)=ct M^{\a,\b}f(x)$ is
the solution, then a similar Strichartz type estimate holds with the parameter $\rho$ and the
corresponding assumptions and constraints adjusted suitably.

When $\a=n/2-1$, $\b > -n/2+1/2$, $n=1,2,\ldots$, and $v(x,t) = c M_t^{\a,\b}f_0(|x|)$,
$(x,t) \in \mathbb{R}^n \times \mathbb{R}_+$ is a spatially radial (weak) solution to a PDE problem with
a radial $f=f_0(|\cdot|)$ entering an initial value condition, then Theorem \ref{thm:main} together with
Remark \ref{rem:min} establish the weighted Strichartz type estimate
$$
\Big\| \big\| v(x,t) |x|^{-B} \big\|_{L^q(\mathbb{R}^n,dx)} \Big\|_{L^r(t^{\rho}dt)} \lesssim
	\big\| f(x) |x|^A \big\|_{L^p(\mathbb{R}^n,dx)},
$$
provided that the parameters satisfy all the restrictions imposed by Theorem \ref{thm:main} and $q \le r$.
If $v(x,t) = ct M_t^{\a,\b}f_0(|x|)$ happens to be such a solution, then again one infers a Strichartz type
estimate by taking $\tilde{\rho} = {\rho} + r$ instead of $\rho$.

We now give examples of Cauchy initial value problems for several classical PDEs, where solutions
$u(x,t)$ or $v(x,t)$ of the above form indeed occur.
\begin{itemize}
\item[(I)] \textbf{Euler-Poisson-Darboux equation}.
Let $n \ge 1$ and $\Box_{\b}$ be the EPD operator related to $\mathbb{R}^n$,
$$
\Box_{\b}v = \Delta_x v - v_{tt} - \frac{n+2\b-1}{t} v_t.
$$
Consider the Cauchy problem in $\mathbb{R}^n\times \mathbb{R}_+$
\begin{equation} \label{EPD}
\Box_{\b} v =0, \qquad v(x,0) = f(x), \qquad v_t(x,0) = 0,
\end{equation}
with a radial initial position $f=f_0(|\cdot|)$. Then, assuming that $\b > -n/2+1/2$,
$$
v(x,t) = M_t^{n/2-1,\b}f_0(|x|)
$$
is a solution to the singular Cauchy problem \eqref{EPD}; see \cite{We,Brest,Rubin}.
Note that here the special case $\b = 2/3-n/2$ corresponds to the Tricomi equation.

\item[(II)] \textbf{Wave equation}.
Let $n \ge 1$ and observe that $\Box_{(1-n)/2}$ is the wave operator related to $\mathbb{R}^n$.
Consider the Cauchy problem in $\mathbb{R}^n\times \mathbb{R}_+$
\begin{equation} \label{wave}
\Box_{(1-n)/2} v = 0, \qquad v(x,0) = 0, \qquad v_t(x,0) = f(x),
\end{equation}
with radial initial speed $f=f_0(|\cdot|)$. Then, see \cite{Stein},
$$
v(x,t) = tM_t^{n/2-1,(3-n)/2}f_0(|x|)
$$
is a solution to \eqref{wave}.

\item[(III)] \textbf{Bessel EPD and wave equations}.
For $\a > -1$, let $L_{\a}$ be the one-dimensional Bessel operator
$$
L_{\a} = \frac{d^2}{dx^2} + \frac{2\a+1}{x}\frac{d}{dx}.
$$
When $\a = n/2-1$, $n=1,2,\ldots$, $L_{\a}$ is the radial part of the standard Laplacian in $\mathbb{R}^n$.
Let us consider the following differential problems in $\mathbb{R}_+\times \mathbb{R}_+$:
\begin{align} \label{BEPD}
L_{\a}u - u_{tt} - \frac{2\a+2\b+1}{t} u_t = 0, \qquad & u(x,0) = f(x), \qquad u_t(x,0)=0, \\
L_{\a}u - u_{tt} = 0, \qquad & u(x,0)=0, \qquad u_t(x,0) = f(x). \label{Bwave}
\end{align}
Then, for $\b > -\a-1/2$,
$$
u(x,t) = M_t^{\a,\b}f(x)
$$
is a solution to \eqref{BEPD}, while
$$
u(x,t) = tM_t^{\a,-\a+1/2}f(x)
$$
is a solution to \eqref{Bwave}. Observe that the Bessel EPD operator appearing in \eqref{BEPD} is in fact
a difference of two Bessel operators, one of them acting on the spatial variable, the other one on the time variable.
\end{itemize}

It is worth pointing out that solutions to \eqref{wave} and \eqref{Bwave} with the initial conditions
reversed, i.e.\ when the initial speed is zero and the initial position is prescribed, express as
\begin{align*}
v(x,t) & = \mathcal{M}_t^{n/2-1,(1-n)/2}f(x),\\
u(x,t) & = t\mathcal{M}_t^{\a,-\a-1/2}f(x).
\end{align*}
The parameters here, however, correspond to the critical line $\a+\b=-1/2$ where $\mathcal{M}_t^{\a,\b}$
becomes a singular integral, the more subtle case that is not treated in this paper.

Another comment concerns connections of the above mentioned solutions with initial positions/speeds.
More precisely, the question is in what sense the solutions converge to initial conditions as
time decreases to $0$, and for what ranges of the parameters the convergence takes place.
In general, without requiring much regularity of the initial data, this is a difficult question that
requires studying time-maximal operators. We refer
to \cite{Stein,MuSe,CoCoSt,DuoMoyuaOrue2} for some partial results.



\end{document}